\documentclass[a4paper,10pt,reqno]{article}
\usepackage[body=16cm,top=3cm,bottom=4cm]{geometry}
 \usepackage{amsrefs}
\usepackage{amsmath,amsfonts,amsthm,amscd,amssymb}
\usepackage[colorlinks=true]{hyperref}

\newcommand{\eps}{\varepsilon}

\makeatletter
\newcommand{\eqcolon}{\mathrel{\mathord{=}\raise.2\p@\hbox{:}}}
\newcommand{\coloneq}{\mathrel{\raise.2\p@\hbox{:}\mathord{=}}}
\makeatother

\newcommand{\dd}{\mathrm{d}}

\newcommand{\RR}{\mathbb{R}}

\newcommand{\LL}{\mathcal{L}}

\newcommand{\CC}{\mathcal{C}}

\newcommand{\ZZ}{\mathbb{Z}}
\renewcommand{\LL}{\mathcal{L}}

\newcommand{\cF}{\mathcal{F}}

\newcommand{\cN}{\mathcal{N}}
\newcommand{\cC}{\mathcal{C}}
\newcommand{\cL}{\mathcal{L}}
\newcommand{\II}{\mathbb{I}}
\newcommand{\TT}{\mathbb{T}}

\newtheorem{theorem}{Theorem}[section]

\newtheorem{definition}[theorem]{Definition}

\newtheorem{hypothesis}[theorem]{Hypothesis}
\newtheorem{lemma}[theorem]{Lemma}

\newtheorem{proposition}[theorem]{Proposition}

\newtheorem{remark}[theorem]{Remark}

\begin{document}
\title{Nonlinear PDEs with modulated dispersion I: \\ Nonlinear Schr\"odinger equations}
\author{K.~Chouk, M.~Gubinelli\footnote{Member of the Institut Universitaire de France.}\\
{\small
CEREMADE  UMR 7534}\\{\small  Universit\'e Paris--Dauphine \& CNRS, France
}\\{\small \texttt{\{chouk,gubinelli\}@ceremade.dauphine.fr}
}}
\date{\today}
\maketitle

\begin{abstract}
We start a study of various nonlinear PDEs under the effect of a  modulation in time of the dispersive term. In particular in this paper we consider the modulated non-linear Schr\"odinger equation (NLS) in dimension 1 and 2 and the derivative NLS in dimension 1. We introduce a deterministic notion of ``irregularity" for the modulation and obtain local and global results similar to those valid without modulation. In some situations, we show how the irregularity of the modulation improves the well--posedness theory of the equations.    
We develop two different approaches to the analysis of the effects of the modulation. A first approach is based on novel  estimates for the regularising effect of the modulated dispersion on the non-linear term using the theory of controlled paths. A second approach is an extension of a Strichartz estimated first obtained by Debussche and Tsutsumi in the case of the Brownian modulation for the quintic NLS. \\[.5cm]   
\textbf{Keywords:} Dispersion management; Young integrals; Controlled paths; Stochastic Strichartz inequality; Non-linear Schr\"odinger equation; Regularization by noise. 
\end{abstract}

\tableofcontents

\section{Introduction}

With this paper we start a study of nonlinear PDEs of the form
\begin{equation}
\label{eq:abs}
\frac\dd{\dd t} \varphi_t = A \varphi_t \frac{\dd w_t}{\dd t}  + \cN(\varphi_t),\qquad t\ge 0
\end{equation}
where $w:\RR_+\to\RR$ is an arbitrary continuous function, $A$ is an unbounded linear operator and $\cN$ some general non-linearity. We will be concerned by the case where  $A$ is a dispersive  operator like the Schr\"odinger operator $i\partial^2$ or the Airy operator $\partial^3$ acting on periodic or non-periodic functions on $\RR^n$ and where $\cN$ is some polynomial non-linearity possibly with derivative terms. In this paper we will limit our analysis to the following cases:
\begin{enumerate}
\item (NLS) Non-linear cubic Schr\"odinger equation in $\TT^n,\RR^n$, $n=1,2$,  $A=i\partial^2$, $\cN(\phi)=i|\phi|^2\phi$; 
\item The general NLS with polynomial non-linearity in $\RR$,  $A=i\partial^2$, $\cN(\phi)=i|\phi|^\mu\phi$, $\mu\in(1,4]$; 
\item (dNLS) Non-linear (Wick-ordered) derivative cubic Schr\"odinger equation in $\TT$,   $A=i\partial^2$, $\cN(\phi)=i \partial (|\phi|^2-\|\phi\|_{H^0}^2)\phi$. 
\end{enumerate}
In all these cases the Banach space $V$ will belong to the scale of Sobolev spaces $H^\alpha$, $\alpha \in \RR$ defined as the completion of smooth functions with respect to the norm
\begin{equation}
\label{eq:H-norm}
\|\phi\|_{\alpha}= \|\phi\|_{ H^\alpha} = \|\langle \xi \rangle^\alpha \hat \phi(\xi)\|_{L^2(\RR^n)}
\end{equation}
where $\hat \phi$ is the Fourier transform of $\phi:\RR^n\to \CC$ and $\langle \xi\rangle=(1+|\xi|^2)^{1/2}$. Similar definition holds in the periodic case where $\RR^n$ is replaced by $\TT^n$ with $\TT=[0,2\pi[$ with periodic boundary conditions. 

\medskip

The (randomly) modulated NLS equation has been subject of interest in recent literature (for example~\cites{kunze_ground_2005,zharnitsky_stabilizing_2001,abdullaev_soliton_2000,marty2006,hundertmark_decay_2009, deBouard20101300,Debussche2011363, hundertmark_super-exponential_2012}), especially related to  applications to soliton management in optical wave-guides. We were directly inspired by the recent work of De~Bouard and Debussche~\cite{deBouard20101300} who study the Nonlinear Schr\"odinger equation with Brownian modulation. They show that it describes the homogenisation of the deterministic Nonlinear Schr\"odinger Equation with time dependent and ergodic dispersion. Our work can be seen as a generalisation of theirs and of the subsequent results of Debussche and Tsutsumi~\cite{Debussche2011363} to a general class of irregular modulations (which however are not required to be random). 
Aside of specific applications we have two main theoretical motivations:
\begin{itemize}
\item[i)] Understanding the properties of dispersive PDEs in non-homogeneous environments and on what can be expected as far as ``generic" properties of the equation are concerned. Modulated equations seems to rule out classical techniques of Fourier analysis (e.g. Bourgain spaces in the case of KdV) and other important tools like Strichartz estimates require different proofs than those standard in the literature. Many conservation laws are also not available in the modulated context and this affects the analysis of global solutions.   

\item[ii)] The study of the \emph{regularisation effect} of a non-homogene\-ous time modulation in the spirit of the recent work of Flandoli, Priola and one of the authors~\cite{FGP} on the stochastic transport equation 
$$
\frac{\partial}{\partial t} u(t,x) + b(t,x) \cdot \nabla u(t,x) + \nabla u(t,x) \circ \frac{\dd B_t}{\dd t} =0
$$
where the modulation is provided by a Brownian motion $B$. In this context was explicitly shown that the addition of the random transport term improves the well-posedness theory of the equation and provide uniqueness in cases where the deterministic equation has multiple solutions, for example when the vector field $b$ it is only H\"older continuous in space. For a perspective on regularisation by noise in transport equations the reader can refer to~\cites{attanasio_renormalized_2011, flandoli_interaction_2011, fedrizzi_noise_2013, flandoli_remarks_2013, flandoli_random_2011, flandoli_full_2011, DeFlVi, FlMaNe}.

\end{itemize}

\medskip

Dealing with general modulations $w$  in Eq.~\eqref{eq:abs} poses a problem since it is non-trivial to give a  meaning to the derivative of $w$. If $w$ is sampled according to the Wiener measure then the differential equation can be understood via It\^o's stochastic calculus. Actually, De~Bouard and Debussche~\cite{deBouard20101300} observe that interpreting the stochastic differential in Stratonovich sense is the most natural choice in this context since it preserves the mild formulation of the equation (see below) and allows to describe the scaling limit of smooth ergodic modulations.  In the more general situation  the interpretation of eq.~\eqref{eq:abs} as an It\^o or Stratonovich stochastic partial differential equation (SPDE) is not possible. Therefore we prefer to describe solutions directly via a mild formulation putting aside the problem of giving a proper weak--formulation of the equation. If we denote by  $(e^{tA})_{t\in\RR}$ the group of isometries of $V=H^\alpha$ generated by $A$, the mild solution $\varphi$ of eq.~\eqref{eq:abs} is an element of $C(\RR_+;V)$ satisfying formally the equation
\begin{equation}
\label{eq:mild}
\varphi_t = U^w_t \varphi_0 +  U^w_{t} \int_0^t  (U^w_{s})^{-1}\cN( \varphi_s )\dd s, \qquad t\ge 0,
\end{equation}
where $U^w_t = e^{A w_t }$ is the operator obtained by a time-change of the linear evolution associated to $A$ using the function $w$. In this form the equation makes sense for arbitrary continuous function $w$. 

\medskip

The aim of this paper is to analyse eq.~\eqref{eq:mild} under some hypothesis on the ``irregularity" of the perturbation $w$. In particular if $w$ is sufficiently irregular (in a precise sense to be specified below) then we will be able to show that the above nonlinear PDE  can be solved in spaces which are comparable to those allowed by the classical equation
\begin{equation}
\label{eq:abs-deter}
\frac\dd{\dd t} \varphi_t = A \varphi_t  + \cN(\varphi_t),\qquad t\ge 0
\end{equation}
 and that in some situations the combination of the irregularity of the perturbation and the non-linear interaction provides a  regularising effect on the equation. 
 
 \medskip
 
Let us now be more specific about the kind of solutions we are looking for. The nonlinear term $\cN( \varphi_s )$ in eq.~\eqref{eq:mild} does not belong in general to $V$ and more seriously cannot be sensibly defined for arbitrary element of $V$ if $\alpha<0$. To give it a meaning we proceed by approximation. Let $\Pi_N : H^\alpha \to H^\alpha$ be the projector on Fourier modes $|\xi|\le N$: $\widehat{\Pi_N f}(\xi)= \mathbb{I}_{|\xi|\le N} \hat f(\xi)$ where $\hat f$ denotes the Fourier transform of $f\in H^\alpha$ and let $\mathcal{N}_N(\phi)= \Pi_N \mathcal{N}(\Pi_N \phi)$ be the Galerkin regularisation of the non-linearity. 

\begin{definition}
\label{def:solution}
Let $T\in(0,+\infty]$ be a time horizon. Given a function $\varphi\in C([0,T);V)$ we write
\begin{equation}
\label{eq:def-int}
\int_0^t  (U^w_{s})^{-1}  \cN( \varphi_s )\dd s  = \lim_{N\to \infty} \int_0^t  (U^w_{s})^{-1}  \cN_N( \varphi_s )\dd s  .
\end{equation}
whenever the limit exists in $C([0,T);V)$. 

For any $T<+\infty$ a function $\varphi\in C([0,T);V)$  is a (local) solution to~\eqref{eq:mild} in $V$ with initial condition $\phi\in V$ if the above limit exists and eq.~\eqref{eq:mild} is satisfied in $C([0,T);V)$ with $\varphi_0= \phi$.
The solution is global if $T$ can be taken $+\infty$.
\end{definition}

It should be noted that the quantity in eq.~\eqref{eq:def-int} is not a usual Bochner integral but only a convenient notation for the limit procedure. Indeed $\cN( \varphi_s )$ will exist only as a space-time distribution and not as a continuous function with values in $V$. 
 
 \medskip

The next definition concerns the particular notion of ``irregularity" of the perturbation that will be relevant in our analysis and which has been introduced by Catellier and Gubinelli in~\cite{CatellierGubinelli}.

\begin{definition}
\label{def:irregularity}
Let $\rho>0$ and $\gamma>0$. We say that a function $w\in C([0,T];\RR)$ is $(\rho,\gamma)$-irregular if :
$$
\|\Phi^w\|_{\mathcal{W}^{\rho,\gamma}_T} = \sup_{a\in \RR} \sup_{0\le s < t\le T} (1+ |a|)^\rho \frac{\left|\Phi^w_{t}(a)-\Phi^w_{s}(a)\right|}{|s-t|^\gamma} < +\infty
$$
where $\Phi^w_{t}(a)=\int_0^t e^{ia w_r}\dd r$. Moreover we say that $w$ is $\rho$-irregular if there exists $\gamma>1/2$ such that $w$ is $(\rho,\gamma)$-irregular.
\end{definition}
As it is apparent from this definition the notion of irregularity that we need is related to the \emph{occupation measure} of the function $w$~(see for example the review of Geman and Horowitz on occupation densities for deterministic and random processes~\cite{geman_occupation_1980}), in particular to the decay of its Fourier transform at large wave-vectors as measured by the exponent $\rho$. The time regularity of this Fourier transform, measured by the H\"older exponent $\gamma$, will also play an important r\^ole. 

Existence of (plenty of) perturbations $w$ which are $\rho$-irregular is guaranteed by the following theorem, proven in~\cite{CatellierGubinelli} and which has been inspired by some computations in a previous paper of Davie~\cite{Davie}.
\begin{theorem}
Let $(W_t)_{t\ge 0}$ be a fractional Brownian motion of Hurst index $H\in(0,1)$ then for any $\rho < 1/2H$  there exist $\gamma > 1/2$ so that with probability one the sample paths of $W$ are $(\rho,\gamma)$-irregular.
\end{theorem}
In particular there exists continuous paths which are $\rho$-irregular for arbitrarily large $\rho$. Using well known properties of support of the law of the fractional Brownian motion it is also possible to show that there exists $\rho$-irregular trajectories which are arbitrarily close in the supremum norm to any smooth path. It would be interesting to study more deeply the irregularity of continuous paths ``generically". 

\medskip

In our opinion an important general contribution of the present work is the observation that the regularity of the occupation measure of $w$ seems to play a major r\^ole in the understanding of the regularising properties of $w$ in a non-linear PDE context. This in turn prompts the need to  understand more deeply the link of the notion of $\rho$-irregularity with the path-wise properties of $w$. Indeed, apart from the classical contribution of Geman and Horowitz~\cite{geman_occupation_1980}, the authors are not aware of any systematic study of occupation measures from the point of view of their action on spaces of functions, topic which is central to our analysis. Let us explain this better: let 
$$
T^w_{ t} f(x) =   \int_0^t f(x+w_r) \dd r
$$ 
for measurable functions $f:\RR\to\RR$. Then $T^w_{ t} (e^{i a\cdot})(x) =  \Phi^w_{ t}(a) e^{i ax}$ which shows for example that if $w$ is $(\rho,\gamma)$-irregular then
$$
\|T^w_{t} f-T^w_{s} f\|_{H^\rho(\RR)} \lesssim |t-s|^\gamma \|f\|_{H^0(\RR)}
$$
meaning that $T_{t}^w$ is a regularising operator. This point of view links our research to the topic of improving bounds for averages along curves (see for example the paper of Tao and Wright~\cite{tao_lp_2003} and also that of Pramanik and Seeger~\cite{pramanik}) and to the  averaging lemmas for  kinetic formulation of stochastic conservation laws recently studied by Lions, Perthame and Souganidis~\cite{lions_stochastic_2012}. And of course to the work of Debussche and Tsutsumi~\cite{Debussche2011363} on Strichartz estimates for the Brownian modulation which will be generalised below in the context of irregular modulations. 

As we have already said, Catellier and Gubinelli  introduced the notion of $\rho$-irregularity in~\cite{CatellierGubinelli} and started a study of the behaviour of the averaging operator $T^w$ for random paths in the context of the regularisation by noise phenomenon for ODE with irregular drift. However much is still not very well understood. 
For example, it remains an open problem  to understand  what happens if we replace $w$ with a regularised version $w^\eps$. Alternatively we could imagine to add to $w$ some ``perturbation'' which does not change its local behaviour.  
In this respect we conjecture that if $w$ is $(\rho,\gamma)$-irregular then for any smooth function $\varphi$ the perturbed path $w^\varphi = w+\varphi$ is still $(\rho,\gamma)$-irregular but we are only able to prove this in the specific situation where $w$ is a fractional Brownian motion and $\varphi$ is a deterministic perturbation, or more generally but with a loss of $1/2$ in the $\rho$ irregularity of $w^\varphi$: both results (with precise statements) are obtained in~\cite{CatellierGubinelli}. In the case of a smooth $w$ we have the following straightforward result:
\begin{proposition}
Let $w:[0,T]\to\RR$ be a twice differentiable function such that $c_T=\inf_{t\in[0,T]}|w'_t|>0$ for any $T>0$ and $\frac{w''}{(w')^2}\in L^1_{loc}(0,+\infty)$ then $w$ is $(1-\gamma,\gamma)$ irregular for all $\gamma\in(0,1)$.
\end{proposition}
\begin{proof}
Integration by parts gives
$$
ia(\Phi^w_{t}(a)-\Phi^w_{s}(a))=\frac{e^{iaw_t}-e^{iaw_{s}}}{w'_t}+\int_{s}^{t}(e^{iaw_\sigma}-e^{iaw_s})\frac{w_\sigma''}{(w_\sigma')^2}\dd\sigma
$$
and the result follow immediately from the hypothesis.
\end{proof}
A simple but still remarkable fact about  $(\rho,\gamma)$-irregular functions is given by the following theorem which is an original contribution of the present work.
\begin{theorem}
Let $w$ be a $\delta$-H\"older function on $[0,T]$ then for every $\gamma,\rho>0$ such that $\gamma+\delta>1$ and $\rho>{(1-\gamma)}/{\delta}$ we have that $||\Phi^w||_{\mathcal W_T^{\gamma,\rho}}=+\infty$
\end{theorem}
\begin{proof}
Let us assume that $w:[0,1]\to \RR$ is $(\rho,\gamma)$-irregular. A simple computation then gives
$$
e^{ia}(t-s)=\int_{s}^te^{ia(1- w_\sigma)}e^{ia w_\sigma}\dd\sigma=\int_s^te^{ia(1- w_\sigma)}\dd\Phi^w_\sigma(a)
$$ 
 where the integral in the r.h.s. is understood in the Young sense according to  Theorem~\eqref{th:young} below. From the Young integral estimates and the $(\rho,\gamma)$-irregularity hypothesis we readily obtain that
$$
|t-s|=|e^{ia}(t-s)|\lesssim ||\Phi^w||_{\mathcal W_T^{\gamma,\rho}} |t-s|^\gamma (1+|a|)^{\eps-\rho}
$$
for $(1-\gamma)/\delta<\eps<\rho$ which is obviously cannot be true for $|a|$ sufficiently large and this allows to conclude that $||\Phi^w||_{\mathcal W_T^{\gamma,\rho}} =+\infty$.
\end{proof}
\begin{remark}
What is nice about this theorem is the fact that upper bounds on $\rho$ gives lower bounds on the H\"older index and
a simple consequence is that there exist no $\rho$-irregular Lipschitz function if $\rho>1/2$. 
\end{remark}

\bigskip
To deal with NLS type equations with irregular modulations in the sense of Definition~\ref{def:irregularity} we develop two different techniques which give complementary results:
\begin{enumerate}
\item \emph{Controlled paths approach.} We use the idea of controlled paths introduced  by Gubinelli~\cite{GubinelliKdV} to analyse  the periodic KdV equation in negative Sobolev spaces (and more general Fourier-Lebesgue spaces) without relying on Bourgain spaces and the time-homogeneity of the equation. This work has connection to the normal form analysis of Babin, Ilyin and Titi of the same equation~\cite{babin-kdv}.
\item \emph{Modulated Strichartz estimates.} Debussche and Tsutsumi~\cite{Debussche2011363} prove stochastic Strichartz estimates for the  Schr\"odinger semigroup with Brownian modulation. The notion of irregular paths allows to generalise their approach to a wide class of modulations and to exploit the Strichartz estimates to solve the NLS on $\RR$ with general non-linearity up to the quintic case which is critical for the non-modulated NLS. 
\end{enumerate}

Let us now summarise the main contributions of this paper. All along which we are going to make the following basic assumption:
\begin{hypothesis}
\label{hyp:main}
The function $w$ is $(\rho,\gamma)$-irregular for some $\rho >0$ and $\gamma>1/2$.
\end{hypothesis}
Our first  result is about the modulated cubic (NLS) equation.
\begin{theorem}
\label{th:nls-global}
Assume that $\rho>1/2$. Then the modulated cubic NLS equation on $\TT$ and $\RR$ has a  global solution in $H^\alpha$ for any $\alpha\ge 0$. Uniqueness holds in the subspace $\mathcal{D}^w(H^\alpha)\subset C(\RR_+,H^\alpha)$ (defined below) and the flow is locally Lipschitz continuous in $\mathcal{D}^w(H^\alpha)$ endowed with a suitable norm.
\end{theorem}
We remark that the globals existence result for the cubic modulated NLS in $H^\alpha$ $\alpha\in(0,1)$ are a novel result as the Strichatrz  estimates obtained by Debussche and Debouard  
in \cite{deBouard20101300} allow only to obtain global existence for $\alpha\geq1$ or $\alpha=0$. Moreover the global wellposedness on the torus is also a novel fact due to the fact that the Strichartz estimates obtained in~\cite{deBouard20101300} to solve the equation are only available for the real line, and is also relevant to remark that even in the regular case $w_t=t$ the global existence result below $H^1$ is obtained by Bourgain in~\cite{bour1} using the Fourier truncation method.
In the case of Brownian modulation the global solution for $\alpha=0$ has already been constructed by de~Bouard et Debussche~\cite{deBouard20101300}. Here we extend their result to any $\alpha \ge 0$ and any sufficiently irregular modulation. Global solutions for any $\alpha \ge 0$ are the result of the $L^2$ conservation law and some regularity preservation estimates for the non-linear term. The possibility of proving existence of global solution with positive regularity is a regularity preservation phenomenon which seems to be due to the irregular character of the perturbation. In this sense we can consider this as another instance of the regularisation by noise phenomenon. 

Let us now recall what happens when $w_t=t$, Tsutsumi~\cite{tsu} proved the global well-posedness of the cubic NLS in $L^2(\mathbb R)$. His proof is based on the dispersive properties of the linear Schr\"odinger operator expressed by the Strichartz estimates and the conservation of the $L^2$-norm. For the periodic problem on the torus, Bourgain in~\cite{bour1} introduced the family of $X^{s,b}$ spaces and used them to prove the global well-posedness of cubic NLS in $L^2(\mathbb T)$. His argument is based on a periodic $L^4$ Strichartz estimates and the conservation of the $L^2$-norm. On $\mathbb R$ Kenig--Ponce--Vega prove in \cite{kpv} the failure of the uniform continuity of the flow in $H^\alpha$ for $\alpha<0$ in the focusing case and the same result is obtained by Christ-Colliander-Tao~\cite{cct} for the defocusing case. In the periodic setting the corresponding results are obtained by N. Burq, P. G\'erard and N. Tzvetkov in~\cite{bgt}.

The techniques we use for NLS work for other similar models. In particular in this paper we consider also the following situations for which we obtain partial results.

\begin{theorem}
\label{th:nse-others} 
\begin{enumerate}
\item If $\rho > 1/2$ the modulated cubic NLS equation on $\mathbb R^2$ has a unique local solution in $H^{\alpha}$ if $\alpha\geq 1/2$;
\item If $\rho > 1$ the modulated dNLS equation on $\mathbb T$ has  a unique local solution in $H^{\alpha}$ if $\alpha\geq 1/2$.
\end{enumerate}
\end{theorem}
A key argument in the proof of all these results is the use of explicit computations allowed by the polynomial character of the non-linearity. These results are however limited to  modulations irregular enough.
 An open problem is to fill the gap between regular and irregular modulations.  One interesting by-product of the controlled approach is the existence of an Euler scheme to approximate the solutions with explicit control of the convergence rate. 

\medskip
A completely different line of attack to the modulated Schr\"odinger equation  comes from the application of the following Strichartz type estimate which can be proved under the same $\rho$--irregularity assumption of Hypothesis~\ref{hyp:main}: 

\begin{theorem}\label{theorem:strichartz-w}
Let $A=i\partial^2_x$, $T>0$, $p\in(2,5]$, $\rho>\min(\frac{3}{2}-\frac{2}{p},1)$ then there exists a finite constant $C_{w,T}>0$  and $\gamma^{\star}(p)>0$ such that the following inequality holds:
$$
\left|\left|\int_{0}^{.}U^{w}_{.}(U_{s}^{w})^{-1}\psi_s\dd s\right|\right|_{L^{p}([0,T],L^{2p}(\mathbb R))}\leq C_{w}T^{\gamma^{\star}(p)}||\psi||_{L^1([0,T],L^2(\mathbb R))}
$$
 for all $\psi\in L^1([0,T],L^2(\mathbb R))$.
 \end{theorem}

As an application we obtain global well--posedness for the modulated NLS equation with generic power nonlinearity i.e.: $\cN(\phi)=|\phi|^{\mu}\phi$: 

\begin{theorem}\label{th:Strich-existence}
Let $\mu\in(1,4]$, $p=\mu+1$, $\rho>\min(1,3/2-\frac{2}{p})$ and $u^0\in L^{2}(\mathbb R)$ then there exists $T>0$  and  a unique $u\in L^{p}([0,T],L^{2p}(\mathbb R)) $ such that the following equality holds:
$$
u_t=U^{w}_tu^{0}+i\int_{0}^{t}U^{w}_t(U^{w}_{s})^{-1}(|u_s|^{\mu} u_s)\dd s
$$
for all $t\in[0,T]$. Moreover we have that $||u_t||_{L^2(\mathbb R)}=||u_0||_{L^2(\mathbb R)}$ and then we have a global unique solution $u\in L_{loc}^p([0,+\infty),L^{2p}(\mathbb R))$ and $u\in C([0,+\infty),L^2(\mathbb R))$. If $u^0\in H^1(\mathbb R)$ then $u\in C([0,\infty),H^{1}(\mathbb R))$.
\end{theorem}
 
We point out that all our techniques are deterministic and that they provide novel results even in the stochastic context, for example when $w$ is taken to be the sample path of a fractional Brownian motion. 
In the Brownian case and when $p<5$ it is possible to show that the solution obtained in Theorem~\ref{th:Strich-existence} coincides a.s. with the solutions obtained by de~Bouard and Debussche~\cite{deBouard20101300} (see Lemma~\ref{lemma:reg} for  details). These solutions corresponds to limits of solutions of Stratonovich type SPDEs. In this context our methods give also the existence of a continuous flow map for the limiting SPDE.

\bigskip

\textbf{Plan.} In Sect.~\ref{sec:controlled} we illustrate the controlled path approach to solution to modulated semilinear PDEs. This approach relies on a non-linear generalisation of the Young integral~\cite{[Young-1936],Lyons1998,[Gubinelli-2004]} for which we provide complete proofs in Sect.~\ref{sec:young}. Using the non-linear Young integral we define and solve Young-type differential equations in Sect.~\ref{sec:sol}. This will  provide a general theory for the constructions  and approximation of the controlled solutions. In Sect.~\ref{sec:reg} we verify that all our models satisfy the hypothesis to apply the general theory we outlined in the previous section and prove the  global existence for the modulated cubic (NLS) equation in $H^\alpha$ for all $\alpha\ge 0$. Finally Sect.~\ref{sec:strich-NLS} is dedicated to proving the Strichartz estimate of Thm.~\ref{theorem:strichartz-w} and apply it to the study the modulated NLS equation on $\RR$ with general non-linearity without relying on controlled solutions.

\medskip
\textbf{Notations.} If $V,W$ are two Hilbert spaces we let $\cL_n(V,W)$ be the Banach space of bounded operators on $V^{\otimes n}$ (considered with the Hilbert tensor product) with values in $W$ and endowed with the operator norm and set $\cL_n(V)=\cL_n(V,V)$. We let $T>0$ denote a fixed time and $C^\gamma([0,T],V)$ the space of $\gamma$-H\"older continuous functions form $[0,T]$ to $V$ endowed with the semi-norm
$$
\|f\|_{C^\gamma([0,T],V)} = \sup_{0\le s < t \le T} \frac{\|f_t-f_s\|_v}{|t-s|^\gamma}
$$
and by $C([0,T],V)$ the space of continuous functions with the sup norm. We will often write $f_{s;t}=f_t-f_s$ for the time increments of functions with values in vector spaces. The notation $J_{s,t}$ will be used instead for the evaluation of functions $J:[0,T]^2\to V$ depending on two times (i.e. not necessarily increments of function $[0,T]\to V$).
 If $V$ is a Banach space then $\mathrm{Lip}_M(V)$ will denote the Banach space of locally Lipshitz map on $V$ with polynomial growth of order $M\ge 0$, that is maps $f:V\to V$ such that
 $$
 \|f\|_{\textrm{Lip}_M(V)} = \sup_{x,y\in V} \frac{\|f(x)-f(y)\|_V}{\|x-y\|_V (1+\|x\|_V+\|y\|_V)^M}<+\infty .
 $$

\section{Controlled paths approach}\label{sec:controlled}

The approach we will use in proving Theorems~\ref{th:nls-global} and~\ref{th:nse-others} is based on ideas coming from the theory of controlled rough paths~\cite{[Gubinelli-2004],gubperime} which have been already used in a variety of contexts:
\begin{enumerate}
\item alternative formulation of rough path theory with the related applications to stochastic differential equations and in general to differential equations driven by non-semimartingale noises~\cites{gubinelli_young_2006,gubinelli_rough_2010,DGT12};
\item approximate evolution of three dimensional vortex lines in incompressible fluids where the initial condition is a non-smooth curve~\cite{bessaih_evolution_2005,brzezniak_global_2010}
\item study of the stochastic Burgers equation (multi-dimensional target space and various kind of robust approximation results)~\cites{hairer_rough_2011,gubperime};
\item definition of controlled (or energy, or martingale) solutions  for a class of SPDEs including the Kardar--Parisi--Zhang (KPZ) equations~\cite{gubinelli_regularization_2012};
\item Hairer's work on the well--posedness and uniqueness theory for the KPZ equation~\cite{hairer_solving_2011};
\item Recently the controlled path approach has also been used to highlight the regularisation by noise phenomenon in ODE with irregular additive perturbations~\cite{CatellierGubinelli} where techniques very similar to those used in this paper were first introduced: in particular the notion of $\rho$-irregularity and the non-linear Young integral.
\end{enumerate}

 Controlled paths are functions which ``looks like" some given reference object. In the case of eq.~\eqref{eq:mild} it looks quite clear that the solution should have the form $\varphi_t = U^w_t \psi_t$ for $\psi_t$ another continuous path in $V$ such that $\varphi_0 = \psi_0$. If we stipulate that $\psi$ has a nice time behaviour then $\varphi$ is somehow "following" the flow of a free solution of the linear equation, modulo a time-dependent slowly varying modulation. The space of controlled paths $\mathcal{D}^w$ (to be defined below) in which we will set up the equation will then be given by functions $\varphi$ such that an H\"older condition holds for $\psi_t = (U^w_{t})^{-1}\varphi_t$. Note that this space depends on the modulation and that different driving functions $w$ and $w'$ would give rise a priory to different spaces $\mathcal{D}^w$ and $\mathcal{D}^{w'}$ of controlled functions. This difference is somehow crucial and make the spaces of controlled paths to be more effective in the analysis of the non-linearities. Let us try to explain why. Assume that $\varphi$ is the simplest path controlled by $w$, that is the solution of the free evolution $\varphi_t = U_t^w \phi$ for some fixed $\phi\in V$ (i.e. not depending on time). In this case the non-linear term in eq.~\eqref{eq:mild} takes the form
$$
\phi_t= U_t \int_0^t  (U^w_{s})^{-1}\cN( U^w_s \phi) \dd s= U_t X_t(\phi)
$$
where  $X:\RR_+ \times V\to V$ is the time-inhomogeneous map given by
\begin{equation}
\label{eq:X}
X_t(\phi) =  \int_0^t  (U^w_{s})^{-1}\cN( U^w_s \phi) \dd s
\end{equation}
We will show that, in the specific settings we will consider, it is possible to actually prove the following regularity requirement:
\begin{hypothesis}
\label{hyp:X}
The map $X$ belongs to $C^\gamma([0,T];\mathrm{Lip}_M(V))$
for some $\gamma > 1/2$ and  $M\ge 0$. 
\end{hypothesis}

In this situation we see that $\phi_t$ is a controlled path such that $\Psi_t = (U^w_t)^{-1} \Phi_t$ belongs at least to $C^{1/2}([0,T],V)$. If we want a space of controlled paths stable under the fixed point map 
$$
\Gamma(\varphi)_t = U^w_t \varphi_0 + U^w_t \int_0^t  (U^w_{s})^{-1}\cN( \varphi_s) \dd s
$$
we have to require $t\mapsto (U^w_t)^{-1}\Gamma(\varphi)_t$ to be at least in $C^{1/2}([0,T],V)$ since otherwise even the first step of the Picard iterations will take us out of the space. These considerations suggest the following   definition of controlled paths:
\begin{definition}
\label{def:controlled-paths}
The space of paths $\mathcal{D}^w(V)$ controlled by $w$ consists of all paths $\varphi\in C([0,T],V)$ such that $t \mapsto \varphi^w_t =  (U^w)^{-1}_t \varphi_t$ belongs to $C^{1/2}([0,T],V)$. 
\end{definition}

At this point it is still not clear that the non-linear term is well defined for every controlled paths. Hypothesis~\ref{hyp:X} ensures that the non-linearity is well defined when the controlled path $\varphi$ is such that $\varphi^w$ is constant in time. To allow for more general controlled paths we consider a smooth (in space and time) path $f$: in this case the following computations can be easily justified in all the models we will consider:
$$
 \int_0^t  (U_{s}^w)^{-1}\cN( U^w_s f_s) \dd s  = \int_0^t \left[\frac{\dd}{\dd s}X_{s} \right](f_s) \dd s = \int_0^t X_{\dd s}(f_s) .
$$ 
where the last integral in the r.h.s. should be interpreted as the limit of suitable Riemann sums:
$$
\int_0^t X_{ds}(f_s) := \lim_{|\Pi_{0,t}| \to 0} \sum_{i} X_{ t_i  ;t_{i+1}}(f_{t_i}).
$$
where recall that $X_{s;t}=X_t-X_s$. A key observation is that the map $f \mapsto \int_0^t X_{\dd s}(f_s)$ can be extended by continuity to all the functions $f\in C^{1/2}(V)$ using the theory of Young integrals, indeed note that $X$ is a path of Lipshitz maps with H\"older regularity $\gamma >1/2$ and that this is enough to integrate functions of H\"older regularity $1/2$ since the sum of these two regularities exceed $1$. 
Since the kind of Young integral we use is not  standard we will provide  proofs and estimates in a self-contained fashion below. 
The time-integral of the non-linearity (even if not the non-linearity itself) results then to be a  well defined space distribution for all controlled paths and it is explicitly given by a Young integral involving the modulated operator $X$. We can  indeed recast the mild equation~\eqref{eq:mild} as a Young-type differential equation for controlled paths:
\begin{equation}
\label{eq:Young}
\varphi^w_t = \varphi_0 + \int_0^t X_{\dd s}(\varphi^w_s) .
\end{equation}
Any solution of this equation corresponds to a controlled path $\varphi_t = U^w_t \varphi^w_t$ which solves~\eqref{eq:mild} where the r.h.s. should be understood according to the Theorem~\ref{th:young}.

The Young equation~\eqref{eq:Young}
can then be solved, at least locally in time and in a unique way, in $C^{1/2}(\RR_+,V)$ by a standard fixed point argument. In some cases it is also possible to prove the existence of a conservation law implying  $\|\varphi_t\|_V=\|\varphi_0\|_V$  for all $t\ge 0$ and obtain global solutions. Another byproduct of this approach is the existence of a Lipshitz flow map on $V$.


\subsection{The nonlinear Young integral}
\label{sec:young}

Young theory of integration is well known~\cite{[Young-1936],Lyons1998,[LionsStFlour],[FrizVictoir],[Gubinelli-2004]}. Here we introduce a non-linear variant which is not covered by the standard assumptions. For the sake of completeness we derive here the main estimates in our specific context.

\begin{theorem}[Young]
\label{th:young}
Let $f\in C^\gamma([0,T],\mathrm{Lip}_M(V))$ and $g\in C^\rho([0,T], V) $ with $\gamma+\rho>1$ then the limit of Riemann sums
$$
I_t(f,g) = \int_0^t f_{\dd u}(g_u) =  \lim_{|\Pi| \to 0} \sum_{i} ( f_{t_{i+1}}( g_{t_i}) -f_{ t_i}( g_{t_i}))
$$
exists in $V$ as the partition $\Pi$ of $[0,t]$ is refined, it is independent of the partition, and we have
$$
\|I_t(f,g)-I_s(f,g) -(f_t-f_s)(g_s)\|_{V} \le (1-2^{1-\gamma-\rho})^{-1} \|f\|_{C^\gamma([0,T],\mathrm{Lip}_M(V))} \|g\|_{C^\rho([0,T], V)} (1+\|g\|_{C^0([0,T],V)})^M |t-s|^{\gamma+\rho}.
$$
\end{theorem}

\begin{proof}
We give a new proof of this fact. Let $f,g$ be smooth functions in $\mathrm{Lip}_M(V)$ and $V$ respectively. Define the bilinear forms
$
I_{t}(f,g) = \int_0^t  (\dd_u f_u)(g_u)
$ 
and  $J_{s,t}(f,g) = I_{s;t}(f,g)-f_{s;t}(g_s)$ and note that these last satisfy
$
J_{s,t}(f,g) = J_{s,u}(f,g) + J_{u,t}(f,g) + (f_{u;t}(g_u)- f_{u;t}(g_s))
$
for all $s\le u \le t$.
Let $t^n_k = s+(t-s)k2^{-n}$ for $k=0,\dots,2^{n}$. By induction:
$$
J_{s,t}(f,g) = \sum_{i=0}^{2^n-1} J_{t^n_{i},t^n_{i+1}}(f,g) + \sum_{k=0}^n \sum_{i=0}^{2^k-1}
(f_{t^k_{2i+1};t^k_{2i+2}}(g_{t^k_{2i+1}})-f_{t^k_{2i+1};t^k_{2i+2}}(g_{t^k_{2i}}))
$$
Since $f,g$ are smooth $\|J_{t^n_i,t^n_{i+1}}(f,g)\|_{V} \lesssim_{f,g} |t^n_{i+1}-t^n_i|^2\lesssim_{f,g}  2^{-2n}$ so that
$
\|\sum_{i=0}^{2^n-1} J_{t^n_i,t^n_{i+1}}(f,g)\|_{V} \lesssim_{f,g} 2^{-n} \to 0 
$
as $n\to \infty$. Then we can estimate
$$
\|J_{s,t}(f,g)\|_{V} 
\le \sum_{k=0}^\infty  2^{k(1-\gamma-\rho)}\|f\|_{C^\gamma([0,T],  \mathrm{Lip}_M(V))} \|g\|_{\cC^\rho V } (1+\|g\|_{C([0,T], V)})^M
$$
$$
\le  (1-2^{1-\gamma-\rho})^{-1} \|f\|_{C^\gamma([0,T], \mathrm{Lip}_M(V))} \|g\|_{C^\rho([0,T], V) } (1+\|g\|_{C([0,T], V)})^M
$$

Now assume that $f\in C^\gamma([0,T], \mathrm{Lip}_M(V))$ and $g\in C^\rho([0,T], V)$. Then there exists sequences of smooth function $f_n$ and $g_n$ such that $f_n \to f$ in $C^{\gamma'}([0,T], \mathrm{Lip}_M(V))$ and $g_n \to g$ in 
$C^{\rho'}([0,T], V)$ for all $\gamma'<\gamma$ and all $\rho'<\rho$ and moreover such that $\|f_n\|_{ C^\gamma([0,T], \mathrm{Lip}_M(V))} \le \|f\|_{ C^\gamma([0,T], \mathrm{Lip}_M(V))}$ and $\|g_n\|_{C^\rho([0,T], V)} \le \|f\|_{C^\rho([0,T], V)}$ . The above estimate implies the convergence of  $J_{s,t}(f_n,g_n)\to J_{s,t}(f,g)$ in $V$ for all $s,t$. In turn this implies that, by passing to the limit in the estimate we have also
$
\|J_{s,t}(f,g)\|_{V} \le  (1-2^{1-\gamma-\rho})^{-1} \|f\|_{ C^\gamma([0,T], \mathrm{Lip}_M(V))} \|g\|_{C^\rho([0,T], V)} (1+\|g\|_{C([0,T], V)})^M.
$
Which means that we can define $t\mapsto I_t(f,g)$ such that 
$$
I_{s;t}(f,g) = \int_s^t f_{\dd u}(g_u) = f_{s;t}(g_s) + J_{s,t}(f,g)
$$
for any $f\in  C^\gamma([0,T], \mathrm{Lip}_M(V))$ and $g\in C^\rho([0,T], V)$.
Now assume that $\Pi =\{s\le t_0< t_1 < \cdots < t_n \le t\}$ is a partition of $[s,t]$ and denote with
$
S_\Pi = \sum_{i=0}^{n-1} f_{t_i; t_{i+1}}(g_{t_i})
$
the associate Riemann sum. By the above construction we have $ f_{t_i; t_{i+1}}(g_{t_i})=I_{t_{i+1}; t_i}(f,g)-J_{t_{i+1}, t_i}(f,g)$ with 
$
\|J_{ t_i,t_{i+1}}(f,g)\| \lesssim_{f,g} |t_{i+1}-t_i|^{\gamma+\rho}
$
and so
$$
S_\Pi = \sum_{i=0}^{n-1}I_{ t_i; t_{i+1}}(f,g)+\sum_{i=0}^{n-1}J_{ t_i,t_{i+1}}(f,g)=I_{s; t}(f,g) +\sum_{i=0}^{n-1}J_{ t_i,t_{i+1}}(f,g)
$$
moreover
$
\left\|\sum_{i=0}^{n-1}J_{ t_i,t_{i+1}}(f,g)\right\|_{V} \lesssim_{f,g} \sum_{i=0}^{n-1} |t_{i+1}-t_i|^{\gamma+\rho}  \lesssim_{f,g}  |\Pi|^{\gamma+\rho-1} |t-s|
$
which implies that $S_\Pi \to I_{s; t}(f,g)$ as $|\Pi|\to 0$ and the integral which we defined above by the continuous extension of the bilinear forms $I_{s; t}(f,g)$ coincides indeed with the limit of Riemann sums on arbitrary partitions. 
\end{proof}

\subsection{Young solutions}
\label{sec:sol}

With the estimates of Young integral we can set up a standard fixed point procedure to prove existence of local solution and their uniqueness.  

\begin{theorem}
\label{thm:young-eq}
Let $M\ge 0$ and $\gamma > 1/2$ and assume that $X\in C^\gamma(\mathrm{Lip}_M(V))$ and  $X_t(0)=0$. For any $\psi_0\in V$  there exists $T>0$ depending only on $\|X\|_{C^{\gamma}([0,T],\mathrm{Lip}_M(V))}$ and $\|\psi_0\|_V$ such that the Young equation
\begin{equation}
\label{eq:young-fixpoint}
\psi_t = \psi_0 +  \int_0^t X_{\dd s}(\psi_s),\qquad 0\le t< T.
\end{equation}
has a unique solution $\psi \in C^{1/2}([0,T];V)$. If we can take $T=+\infty$ we say that the solution is global.
\end{theorem}
\begin{proof}
Fix a finite $T>0$ and define standard Picard's iterations by
$$
\psi^{(n+1)}_t = \psi_0 + \int_0^t X_{\dd s}(\psi^{(n)}_s) ,\qquad 0 \le t \le T
$$
with $\psi^{(0)}_t = \psi_0$. 
Now
$$
\|\int_0^t X_{\dd s}(\psi^{(n)}_s) -X_{t}(\psi_0)\|_V\lesssim T^{\gamma+1/2} \|X\|_{C^{\gamma}([0,T],\mathrm {Lip}_M(V))} (1+\|\psi^{(n)}\|_{\cC^0 V})^M \|\psi^{(n)}\|_{C^{1/2}([0,T],V)}
$$
$$
\lesssim T^{\gamma} \|X\|_{C^{\gamma}([0,T],\mathrm{Lip}_M(V))} (1+\|\psi_0\|_V+T^{1/2}\|\psi^{(n)}\|_{C^{1/2}([0,T],V)})^{M+1} 
$$
and
$$
\|\psi^{(n+1)}\|_{\cC^{1/2}V} \lesssim \|X\|_{C^{\gamma}([0,T],\mathrm{Lip}_M(V))} T^{\gamma} (1+\|\psi_0\|_V+T^{1/2}\|\psi^{(n)}\|_{C^{1/2}([0,T],V)})^{M+1}
$$
which means that for sufficiently small $T$ (depending only on $\|\psi_0\|_V$) we can have $T^{1/2} \|\psi^{(n)}\|_{C^{1/2}([0,T],V)} \le 1$ for all $n\ge 0$. Moreover in this case
$
\|\psi^{(n+2)}-\psi^{(n+1)}\|_{C^{1/2}([0,T],V)} \lesssim_{\|\psi_0\|_V} \|X\|_{C^{\gamma}([0,T],\mathrm{Lip}_M(V))} T^{\gamma-1/2}\|\psi^{(n+1)}-\psi^{(n)}\|_{C^{1/2}([0,T],V)}
$
which for $\|X\|_{C^{\gamma}([0,T],\mathrm{Lip}_M(V))} T^{\gamma-1/2} \lesssim_{\|\psi_0\|_V}1/2$ implies that $(\psi^{(n)})_{n\ge 0}$ converges in $C^{1/2}([0,T],V)$ to a limit $\psi$. By continuity of the Young integral and of the operator $X$ this limit satisfies the equation~\eqref{eq:young-fixpoint}.
This solution exists at least until $t\le T$ where $T$ depends only on the norm of $X$ and $\|\psi_0\|_{V}$. Note that a posteriori $\psi$ belongs to $C^\gamma([0,T],V)$ and not only to $C^{1/2}([0,T],V)$. Uniqueness in $C^{1/2}([0,T],V)$ is now obvious.
\end{proof}
 Of course if $M=0$ it is easy to prove that the existence time $T$ of the local solution does not depend on $\|\phi_0\|_V$ and this implies the existence of solution on arbitrary intervals. In the general case we need further assumptions on the properties of $X$:
\begin{lemma}
\label{lemma:global}
Assume that the hypothesis of Thm.~\ref{thm:young-eq} hold and  that for all $\phi\in V$ such that $\|\phi\|_V\le A$  we have 
 $$|\|\phi+X_{s;t}(\phi)\|_V-\|\phi\|_V|\lesssim C_A |t-s|^{\rho}$$ 
 where $\rho > 1$, then for any $T>0$ and any local solution $\psi$ on $[0,T]$ of the Young equation~\eqref{eq:young-fixpoint} we have $\|\psi_t\|_V = \|\psi_0\|_V$ for all $t\in[0,T]$. This implies that there exists a unique global solution of the Young equation~\eqref{eq:young-fixpoint}.
\end{lemma}
\begin{proof}
Let $M_t = \|\psi_t\|_{V}$. By definition we have that
$$
\psi_t=\psi_s+\int_s^tX_{\dd\sigma}(\psi_\sigma)=\psi_s+X_{s,t}(\psi_s)+R_{s,t}
$$ 
with $R$ is the remainder term of the Young integral given in Thm.~\ref{th:young} and therefore satisfies $|R_{s,t}|\lesssim |t-s|^{2\gamma}$. An easy computation then gives 
$$
|M_t - M_s| = |\|\psi_s + X_{s,t}(\psi_s) + R_{s,t}\|_{V} -\|\psi_s\|_V|
$$
$$
\leq|\|\psi_s + X_{s,t}(\psi_s) \|_{V} -\|\psi_s\|_V|+|\|\psi_s + X_{s,t}(\psi_s) + R_{s,t}\|_{V} -\|\psi_s+X_{st}(\psi_s)\|_V|
$$
$$
\leq |\|\psi_s + X_{s,t}(\psi_s) \|_{V} -\|\psi_s\|_V|+\| R_{s,t}\|_{V}.$$
By our assumptions and by the Young estimates on $R$ we have
$$
|M_t - M_s| \lesssim_{\|\psi_0\|_V} |t-s|^\rho + |t-s|^{1/2+\gamma}.
$$
This relation implies that $M_t$ must be a constant function since  $\rho>1$ and $1/2+\gamma > 1$. Then  $M_t = M_0$ for all $t<T$. The conservation of the $V$ norm allows then to extend the local solution to an arbitrary interval and obtain a global solution satisfying the conservation law.
\end{proof}

\subsection{An Euler scheme for Young equations}
Young equations allow for a straightforward Euler approximation scheme. Let $\psi\in  C^{\gamma}([0,T],V)$ a local solution of the Young equation in $[0,T]$. For any $n\geq 0$ let $\psi_0^n=\psi_0\in V$ and define recursively 
$$
\psi_i^n=\psi_{i-1}^n+X_{\frac{i-1}{n};\frac{i}{n}}(\psi_{i-1}^n) \qquad i={1,...,\lfloor nT\rfloor }
$$
\begin{theorem}
For $n\geq0$ and $0\leq i\leq nT$ let $\Delta^n_i=\psi_i^n-\psi_{\frac{i}{n}}$ then 
$$
\max_{0\leq i<j\leq \lfloor nT\rfloor}\frac{|\Delta_j^n-\Delta_i^n|}{|i-j|^{\gamma}}=O(n^{1-2\gamma}).
$$
\end{theorem} 
\begin{proof}
We remark that $\psi_j^n-\psi_i^n=\sum_{l=i}^{j-1}X_{\frac{l}{n};\frac{l+1}{n}}(\psi_l^n)$ and for $0\leq i<j\leq \lfloor nT\rfloor$ and define the partition of $[i/n,j/n]$ by $\pi^{j-i+1}=(t_k^n)_{i\leq k\leq j}$ with $t^n_k=\frac{k}{n}$. Let  $M_{ij}^{\pi^{j-i+1}}=\sum_{l=i}^{j-1}X_{t_l^n;t_{l+1}^n}(\psi_l^n)$ and consider the partition $\pi^{j-i}=\pi^{j-i+1}-\{t^n_k\}$ for $i<k<j$. Then  
$$
M^{\pi^{j-i+1}}_{ij}=M^{\pi^{j-i}}_{ij}+X_{t^n_k;t^n_{k+1}}(\psi_k^n)-X_{t_k^n;t_{k+1}^n}(\psi_{k-1}^n),
$$
and by induction we obtain immediately that 
$$
\psi_j^n-\psi_i^n=X_{t_i^n;t_j^n}(\psi_i^n)+\sum_{k=i+1}^{j-1}[X_{t^n_k;t^n_{k+1}}(\psi_k^n)-X_{t_k^n;t_{k+1}^n}(\psi_{k-1}^n)]
$$
for all $0\leq i<j\leq \lfloor nT\rfloor$. Let $p_{lk}^{n}=X_{t_l^n;t_k^n}(\psi_q^n)$ and $q_{lk}^n=X_{t_l^n;t_k^n}(\psi_{q/n})$. Using that $\psi$ satisfies the Young equation~\eqref{eq:young-fixpoint}
 we obtain  
$$
\psi_{j/n}-\psi_{i/n}=q_{ij}^n+\sum_{k=i+1}^{j-1}[X_{t^n_k;t^n_{k+1}}(\psi_{t^n_k})-X_{t^n_k;t^n_{k+1}}(\psi_{t^n_{k-1}})]+R_{ij}^n
$$
where 
$
R_{ij}^n=\sum_{k=i}^{j-1}\int_{t_k^n}^{t_{k+1}^n}X_{\dd\sigma}(\psi_{\sigma})-X_{t_k^n;t_{k+1}^n}(\psi_{t_k^n})
$.
For this term we have the  bound 
$
|R_{ij}^n|\lesssim_{|\psi|_{\gamma}+|\psi|_{0},||X||}(j-i)n^{-2\gamma}
$ by standard Young estimates.
Consider
$$
\Delta_j^n-\Delta_i^n=p_{ij}^n-q_{ij}-R_{ij}^n+\sum_{k=i+1}^{j-1}[X_{t^n_k;t^n_{k+1}}(\psi_k^n)-X_{t_k^n;t_{k+1}^n}(\psi_{k-1}^n)-X_{t^n_k;t^n_{k+1}}(\psi_{t^n_k})+X_{t^n_k;t^n_{k+1}}(\psi_{t^n_{k-1}})]
$$
and let 
$$
B_l^n=\max_{0\leq i<j\leq l}\left(\frac{j-i}{n}\right)^{-1}\left|\Delta_j^n-\Delta_i^n-p_{ij}^n-q_{ij}+R_{ij}^n\right|.
$$
To prove our result is suffices to show that $B_{{\lfloor} nT\rfloor}^n=O(n^{1-2\gamma})$. Observe that when $|i-j|<l$ the sum appearing in the expression of 
$ \Delta_i^n-\Delta_j^n $ can be bounded by $B_{l-1}^n$. In fact we have that 
$$
|X_{t^n_k;t^n_{k+1}}(\psi_k^n)-X_{t_k^n;t_{k+1}^n}(\psi_{k-1}^n)-X_{t^n_k;t^n_{k+1}}(\psi_{t^n_k})+X_{t^n_k;t^n_{k+1}}(\psi_{t^n_{k-1}})|\leq C(\frac{j-i}{n})^{2\gamma}(1+B_{l-1}^n)^{M}(n^{1-2\gamma}+B_{l-1}^n)
$$ 
where $C=C(\psi^0,||X||)$ and 
$$
B_l^n\leq C(1+B_{l-1}^n)^M(B_{l-1}^n+n^{1-2\gamma})(l/n)^{2\gamma-1}.
$$  
When $l=1$ we have that $B_1^n=0$ and the result is clearly true. Now assume that for some $l$ we have that $B_{l-1}^n\leq A$ and  define the increasing  map 
$\theta(x)=(l/n)^{2\gamma-1}(1+x)^{M+1}$. Remark that $\theta(n^{2\gamma-1}B_l^n)\leq n^{2\gamma-1}B_{l-1}^n$. Then if $l/n$ is small enough we have that  $\theta$ admits a  fixed point and that $n^{1-2\gamma}B_l^n\leq A<+\infty$ where we take $A$ is the limit  of the sequence $(x_i)$ defined by $x_{i+1}=\theta(x_i)$ and $x_0=0$. Now is suffice to iterate this argument to prove that the bound holds for all $l\leq\lfloor nT\rfloor$.  
\end{proof}
\section{Young theory of the modulated NLS}
\label{sec:reg}
Let $w$ a $(\rho,\gamma)$-irregular path, the aim of this section is to provide the necessary path--wise estimates on the modulated operator $X$ in various NLS models.

\begin{definition}
Let $n\ge 1$. We say that a $n$-linear operator $X:\mathbb R^+\to\mathcal L_n(V)$ on the Banach space $V$ belongs to $\mathcal{X}_{n,V}^w$ if 
\begin{enumerate}
\item For all $T>0$  we have
$$
|X_{s;t}|_{\mathcal L_n(V)}\leq C(1+\|\Phi^w\|_{\mathcal{W}^{\rho,\gamma}_T})|t-s|^{\gamma}
$$
for $s,t\in[0,T]$ and for some finite constant $C>0$ which does not depend on $w$.
\item If we let $X^L_{t}(\varphi_1,\dots,\varphi_n)= \Pi_L X_{t}(\Pi_L\varphi_1,\dots,\Pi_L\varphi_n)$ then we have
$X^L \to X$ in $ C^{1/2}([0,T],\mathcal L_n(V))$.
\end{enumerate}
And then if we define $X_{t}(\psi):=X_{t}(\psi,\psi,. . .,\psi)$ we see that in this case our operator $X$ satisfy the Hypothesis~\eqref{hyp:X}, that is $X\in C^{\gamma}([0,T]; \mathrm{Lip}_M(V))$ for $M=n-1$ and all $T>0$.    
\end{definition}

Once appropriate bounds are obtained for the relevant $X$ operators, the Young theory of  Section~\ref{sec:sol} gives a complete local well-posedness theory for the equation. Additional benefits of this approach are the convergence of approximations and an Euler scheme to approximate solutions.

\subsection{Periodic cubic NLS equation}

We consider the modulated periodic cubic Nonlinear Schr\"odinger equation for which  $A=i\partial^2$ and $\mathcal N(\phi)=i |\phi|^2 \phi$ in eq.~\eqref{eq:abs}.
 By definition $\dot X$ is the time--dependent trilinear operator given by 
$$
\dot X_t(\psi_1,\psi_2,\psi_3) = U^w_{-t}[ (U_{t}^w \psi_1)^* (U_{t}^w \psi_2) (U_{t}^w \psi_3)]
$$
for all $t\in \RR$ and $\psi_1,\psi_2,\psi_3 \in L^2(\TT)$.
Its Fourier transform reads 
\begin{equation*}
\cF \dot X_{t}(\psi_1, \psi_2, \psi_3) (\xi) = \sum_{\substack{\xi_1,\xi_2,\xi_3 \in \ZZ_0\\ \xi=-\xi_1+\xi_2+\xi_3}}  e^{i w_t (\xi^2+\xi_1^2-\xi_2^2-\xi_3^2)} \, \hat \psi_1(\xi_1)^* \hat \psi_2 (\xi_2) \hat \psi_3(\xi_3) 
\end{equation*}
where $\hat \psi_i = \cF \psi_i$. Note that $\xi^2+\xi_1^2-\xi_2^2-\xi_3^2 =2 (\xi -\xi_2)(\xi - \xi_3)$ under the condition that $\xi=-\xi_1+\xi_2+\xi_3$. Then  $X_t = \int_0^t \dot X_s ds$ has Fourier transform 
\begin{equation}\label{eq:NLS-X}
\hat X_{t}(\psi_1,\psi_2,\psi_3)=\sum_{\star} \psi_1(k_1)^{*}\psi_2(k_2)\psi_3(k_3)\Phi^{w}_{t}(2(k-k_2)(k-k_3))
\end{equation}
where the star under the sum means that we sum over all triples $(k_1,k_2,k_3) \in \ZZ^3$ such that $k_2+k_3=k+k_1$ and $k_2\ne k$, $k_3\ne k$, $k_1k_2k_3\ne0$. 
The operator $X$ is well defined for all $\psi_1,\psi_2,\psi_3 \in L^2(\TT)$ with values in the space of Schwartz distributions on $\TT$. Exploiting the $\rho$-irregularity of $w$ the following proposition determines a better regularity for $X$.
A resonance due to the particular form of the non-linear term and to the fact we are working on the torus appears in the operator $X$. This happens when $\xi=\xi_2$ or $\xi=\xi_3$. In this case indeed we have that the modulation function vanishes: $\xi^2+\xi_1^2-\xi_2^2-\xi_3^2 =0$ and no regularisation can take place. In order to take it into account we decompose the operator $X$ according to $X=X^1+X^2$ where
$$
\dot X^1_t (\psi_1, \psi_2, \psi_3)= \psi_3\langle\psi_1,\psi_2\rangle+\psi_2\langle\psi_1,\psi_3\rangle
$$
and
\begin{equation*}
\cF \dot X^2_{t}(\psi_1, \psi_2, \psi_3) (\xi) = \sum_{\substack{\xi_1,\xi_2,\xi_3 \in \ZZ_0\\ \xi=-\xi_1+\xi_2+\xi_3}}   \II_{\xi\neq \xi_2,\xi_3} e^{i w_t (\xi^2+\xi_1^2-\xi_2^2-\xi_3^2)} \, \hat \psi_1(\xi_1)^* \hat \psi_2 (\xi_2) \hat \psi_3(\xi_3) .
\end{equation*}

Then we have

\begin{proposition}
\label{th:reg-1d}
If $\rho>1/2$ then for all $T>0$ and $\alpha\geq 0$ we have $X\in  C^{\gamma}([0,T],\mathcal L_3(H^{\alpha}))$. Moreover  $X=X^1+X^2$ with
$$
X^1_{t}(\psi_1,\psi_2,\psi_3)=t (\psi_2\langle\psi_1,\psi_2\rangle+\psi_3\langle\psi_1,\psi_3\rangle)
$$
and $X^2\in\mathcal{X}_{3,H^\alpha(\mathbb T)}^w$ for all $\alpha>0$. Finally $\sup_{L}||(X^2)^L||_{ C^\gamma([0,T],\mathcal L_3(L^2(\mathbb T)))}<+\infty$.
\end{proposition}
\begin{proof}
Using the Cauchy-Schwartz inequality and the fact that $||\psi||_{L^2}\leq||\psi||_{H^\alpha}$ for $\alpha\geq0$ we obtain immediately the following bound
$$
||X^1_{s;t}(\psi_1,\psi_2,\psi_3)||_{H^\alpha}\leq 2(t-s)||\psi_1||_{\alpha}||\psi_2||_{\alpha}||\psi_3||_{\alpha}
$$
which shows that $X^1 \in  C^{\gamma}([0,T],\mathcal L_3(H^{\alpha}))$ for any $\gamma\le 1$ and $\alpha \ge 0$. 
Let us consider now $X^2$: setting $\Xi = 2 (\xi -\xi_2)(\xi - \xi_3)$ we get
\begin{equation*}
|\langle \psi, X^{2}_{s;t}(\psi_1,\psi_2,\psi_3)\rangle| \le\sum_{\substack{\xi_1,\xi_2,\xi_3 \in\ZZ_0\\ \xi=-\xi_1+\xi_2+\xi_3}} \II_{\xi\neq \xi_2,\xi_3} |\Phi^w_{s;t}(\Xi)| \,|\hat \psi(\xi)^*  \hat \psi_1(\xi_1)^* \hat \psi_2 (\xi_2) \hat \psi_3(\xi_3)|
\end{equation*}
for all $\psi,\psi_1,\psi_2,\psi_3 \in L^2(\TT)$ and all $0<s<t<T$.
By standard application of Cauchy-Schwarz and using the $(\rho,\gamma)$--irregularity norm of $w$ we get
\begin{equation*}
\begin{split}
|\langle \psi, X^{2}_{s;t}(\psi_1,\psi_2,\psi_3)\rangle| & \le (I_{\alpha,\rho})^{1/2} \left( \sup_{a\in\mathbb Z_0}|a|^{\rho}|\Phi^w_{s;t}(a)|\right)  \|\psi\|_{-\alpha} \|\psi_1\|_\alpha \|\psi_2\|_\alpha \|\psi_3\|_\alpha \\
& \le (I_{\alpha,\rho})^{1/2}\|\Phi^w\|_{\mathcal{W}^{\rho,\gamma}_T}  |t-s|^\gamma  \|\psi\|_{-\alpha} \|\psi_1\|_\alpha \|\psi_2\|_\alpha \|\psi_3\|_\alpha
\end{split}
\end{equation*}
with
$$
I_{\alpha,\rho} =
\sup_{\xi\in\ZZ_0}  \sum_{\substack{\xi_1,\xi_2,\xi_3 \in\ZZ_0\\ \xi=-\xi_1+\xi_2+\xi_3}} \II_{\xi\neq \xi_2,\xi_3}  |\xi|^{2\alpha} |\xi_1 \xi_2 \xi_3|^{-2\alpha} |\Xi|^{-2\rho}  .
$$
The finiteness of the constant $I_{\alpha,\beta,\rho}$ is enough to show that $X^2 \in  C^{\gamma}([0,T],\mathcal L_3(H^{\alpha}))$.
Since $\alpha \ge 0$, by using that $ |\xi|^{2\alpha} \lesssim  |\xi_1|^{2\alpha}+|\xi_2|^{2\alpha}+|\xi_3|^{2\alpha}$ we have
$
I_{\alpha,\rho} \lesssim I_{0,\rho}
$
moreover
\begin{equation*}
I_{0,\rho}= \sup_{\xi\in\ZZ_0} \sum_{\xi_2\in\ZZ_0} \II_{\xi\neq \xi_2}  |\xi -\xi_2|^{-2\rho}  \sum_{\xi_3\in\ZZ_0} \II_{\xi\neq \xi_3} |\xi - \xi_3|^{-2\rho} < \infty
\end{equation*}
 provided $\rho > 1/2$. 
Let us now prove that $X^2\in\mathcal X^{w}_{3,H^{\alpha}}$ for $\alpha>0$. Another simple computation shows 
$$
||(X^2)^L_{s;t}-X^2_{s;t}||_{\mathcal L^3H^\alpha} \lesssim (I^L)^{1/2}\|\Phi^w\|_{\mathcal{W}^{\rho,\gamma}_T}  |t-s|^\gamma 
$$
where
\begin{equation*}
\begin{split}
I^L = &\sup_{|\xi|> L}|\xi|^{2\alpha}\sum_{-\xi_1+\xi_2+\xi_3=\xi}|\xi_1 \xi_2 \xi_3|^{-2\alpha}|(\xi-\xi_3)(\xi-\xi_2)|^{-2\rho}
\\&+\sup_{|\xi|\leq L}|\xi|^{2\alpha}\sum_{-\xi_1+\xi_2+\xi_3=\xi,|\xi_1|>L}|\xi_1 \xi_2 \xi_3|^{-2\alpha}|(\xi-\xi_3)(\xi-\xi_2)|^{-2\rho}
\\&+\sup_{|\xi|\leq L}|\xi|^{2\alpha}\sum_{-\xi_1+\xi_2+\xi_3=\xi,|\xi_2|>L}|\xi_1 \xi_2 \xi_3|^{-2\alpha}|(\xi-\xi_3)(\xi-\xi_2)|^{-2\rho}
\\&\equiv I^L_1+I^L_2+I^L_3
\end{split}
\end{equation*}
Now 
\begin{equation*}
\begin{split}
I^L_1\lesssim& \sup_{|\xi|>L}\sum_{\star}|\xi_2 \xi_3|^{-2\alpha}|\xi-\xi_2|^{-2\rho}|\xi-\xi_3|^{-2\rho}+\sup_{|\xi|>L}\sum_{\star}|\xi_1|^{-2\alpha}|\xi_2|^{-2\alpha}|\xi-\xi_2|^{-2\rho}|\xi-\xi_3|^{-2\rho}
\\&\lesssim L^{-\min(2\alpha,2\eps)}
\end{split}
\end{equation*}
for $0<\eps<2\rho-1$ and where the star under the sum mean that $-\xi_1+\xi_2+\xi_3=\xi$. Here  we have used the inequality $|\xi|\lesssim |\xi_2|+|\xi-\xi_2|$ to bound the first sum and for the second sum we have used the fact that $|\xi|\lesssim |\xi_1|+|\xi-\xi_2|+|\xi-\xi_3|$, which gives the needed bound for $I^L_1$. To deal with $I^L_2$ we remark that $|\xi-\xi_2|=|\xi_3-\xi_1|$ and then 
$$
I^L_2\lesssim L^{-2\alpha}\sum_{\xi_2,\xi_3}|\xi_1-\xi_2|^{-2\rho}|\xi_3-\xi_1|^{-2\rho}+L^{-\min(\alpha,\eps)}\sum_{\xi_2,\xi_3}|\xi_3-\xi_1|^{\rho-\eps}|\xi_2-\xi_1|^{\rho}\lesssim L^{-\min(2\alpha,\eps)}.
$$
To obtain the last bound for $I^L_3$ we remark that if $|\xi_2|>L$ the $\max(|\xi_1-\xi_2|;|\xi_1|)\geq L/2$ and then we obtain easily that $I^L_3\lesssim L^{-\min(2\alpha,\eps)}$ which completes the proof.
\end{proof}

\subsection{Cubic NLS equation on $\RR$}
In the case of the cubic NLS on the whole real line the modulated operator $\dot X$ takes the form
\begin{equation*}
\cF \dot X_{t}(\psi_1, \psi_2, \psi_3) (\xi) = \int_{\substack{\xi_1,\xi_2,\xi_3 \in\RR\\ \xi=-\xi_1+\xi_2+\xi_3}} e^{i w_t (\xi^2+\xi_1^2-\xi_2^2-\xi_3^2)} \, \hat \psi_1(\xi_1)^* \hat \psi_2 (\xi_2) \hat \psi_3(\xi_3) \dd \xi_2 \dd \xi_3 .
\end{equation*}
for any $\psi_1, \psi_2, \psi_3 \in L^2(\RR)$. Note that the modulation function  $\Xi= \xi^2+\xi_1^2-\xi_2^2-\xi_3^2$ under the condition  $\xi=-\xi_1+\xi_2+\xi_3$ takes the form $\Xi=2 (\xi -\xi_2)(\xi - \xi_3)$. Then $X$ has the expression
\begin{equation*}
\cF  X_{t}(\psi_1, \psi_2, \psi_3) (\xi) = \int_{\substack{\xi_1,\xi_2,\xi_3 \in\RR^d\\ \xi=-\xi_1+\xi_2+\xi_3}} \Phi^{w}_{t}(\Xi)\, \hat \psi_1(\xi_1)^* \hat \psi_2 (\xi_2) \hat \psi_3(\xi_3) \dd \xi_2 \dd \xi_3 .
\end{equation*}

\begin{proposition}
Assume that  $w$ is  $(\rho,\gamma)$-irregular with $\rho>1/2$ then the modulated operator $X$ associated to the  cubic NLS on $\RR$ satisfies  $X\in C^{\gamma}([0,T],H^\alpha(\RR))$ for all $\alpha\geq0$.
\end{proposition}
\begin{proof}
Let $(\psi_i)_{i=1,2,3}\in H^\alpha$ and $\psi_4\in H^{-\alpha}$ then by a simple computation we have that 
\begin{equation}
\begin{split}
 |\langle\psi_4,X_{s; t}(\psi_1,\psi_2,\psi_3)\rangle|\leq  \int_{\mathbb R^3}&(|x|^{-\alpha}|\hat\psi_4(x)|)|x|^{\alpha}|x_1x_2x_3|^{-4\alpha}|-x_1+x_2+x_3|^{\alpha}\times
\\ & \times \Pi_{i=1,.,3}|x_i|^{\alpha}|\hat\psi_i(x_i)| |\Phi^{w}_{s; t}(2(x_2-x_1)(x_3-x_1))| \dd x_1\dd x_2 \dd x_3
\end{split} \end{equation}
with $x=-x_1+x_2+x_3$ now using the fact that $|x|^{\alpha}\lesssim|x_1|^{\alpha}+|x_2|^{\alpha}+|x_3|^{\alpha}$ and  the Lemma~\ref{lemma=Sch-estimates} below we obtain immediately 
$
|\langle\psi_4,X_{s; t}(\psi_1,\psi_2,\psi_3)\rangle|\lesssim (1+\|\Phi^w\|_{\mathcal{W}^{\rho,\gamma}_T} ) |t-s|^{\gamma}\|\psi_4\|_{-\alpha}\Pi_{i=1,..,3}\|\psi\|_{\alpha}
$.
\end{proof}

\begin{lemma}\label{lemma=Sch-estimates}
Let $(\psi_i)_{i=1,..,4}\in L^2(\mathbb R)$ and define the following integral 
\begin{equation*}
\begin{split}
\mathcal I(\alpha):=\int_{\mathbb R^3}\dd x_1\dd x_2\dd x_3|x_2-x_1|^{\alpha}|\Phi^{w}_{s_1;s_2}(2(x_2-x_1)(x_3-x_1))|
\\ \times|\hat\psi_{1}(x_1)||\hat\psi^*_{2}(x_2)||(\hat\psi_{3})^*(x_3)||(\hat\psi^*_{4})^*(-x_1+x_2+x_3)|
\end{split}
\end{equation*}
then if $w$ is $(\rho,\gamma)$--irregular we have the following bound 
$$
\mathcal I(\alpha)\lesssim (1+\|\Phi^w\|_{\mathcal{W}^{\rho,\gamma}_T} ) |s_2-s_1|^{\gamma}\Pi_{i=1,..4}\|\psi_i\|_{L^2(\mathbb R)}
$$
when $\alpha\in[0,1)$ and $\rho>1/2+\alpha$ or $\alpha=1$ and $\rho>1$.
\end{lemma}
\begin{proof}
 In the case $\alpha<1$ let us split $\mathbb R^{3}=\cup_{i=1,...,4}D_i$ with
$$
D_1=\{(x_1,x_2,x_3)\in\mathbb R^3 ;|x_2-x_1|\geq1,|x_3-x_1|\geq1\},
$$
$$
D_2=\{(x_1,x_2,x_3)\in\mathbb R^3 ;|x_2-x_1|\leq1,|x_3-x_1|\leq1\},
$$
 $$
 D_3=\{(x_1,x_2,x_3)\in\mathbb R^3 ;|x_2-x_1|\leq1,|x_3-x_1|\geq1\},
 $$
 $$
  D_4=\{(x_1,x_2,x_3)\in\mathbb R^3 ;|x_2-x_1|\geq1,|x_3-x_1|\leq1\}.$$
According to this split
 $
 \mathcal I(\alpha)=\sum_{i=1,...,4}I_i
$.
By Cauchy-Schwarz we have $I_l\leq J_l\Pi_{i=1}^{4}\|\psi_{i}\|_{L^2(\mathbb R)}$  for $l\in\{1,2,3\}$ and using the $(\rho,\gamma)$--irregularity of $w$ we have 
\begin{equation*}
\begin{split}
J_1^2&=\sup_{x_1}\int_{\mathbb R^2}\dd x_2\dd x_3 \mathbb I_{\{|x_2-x_1|\geq1;|x_3-x_1|\geq1\}}|x_2-x_1|^{2\alpha}|\Phi^{w}_{s_1; s_2}(2(x_2-x_1)(x_3-x_1))|^2
\\&=\int_{\mathbb R^2}\dd y_2\dd y_3 \mathbb I_{\{|y_2|\geq1;|y_3|\geq1\}}|y_2|^{2\alpha}|\Phi^{w}_{s_1; s_2}(2y_2y_3)|^2
\\&\lesssim \|\Phi^w\|_{\mathcal{W}^{\rho,\gamma}_T}^2 |s_2-s_1|^{2\gamma}(\int_{|y_2|\geq1}\frac{1}{|y_2|^{2\rho-2\alpha}}\dd y_2)(\int_{|y_3|\geq1}\frac{1}{|y_3|^{2\rho}}\dd y_3)<+\infty  
\end{split}
\end{equation*}
when $\rho>\alpha+1/2$.  To bound the term $J_3$ we use again the $(\rho,\gamma)$ irregularity of $w$ and we obtain
\begin{equation*}
\begin{split}
J_3^2&=\sup_{x_1}\int_{\mathbb R^2}\dd x_2\dd x_3 \mathbb I_{\{|x_2-x_1|\leq1;|x_3-x_1|\geq1\}}|x_2-x_1|^{2\alpha}|\Phi^{w}_{s_1; T}(2(x_2-x_1)(x_3-x_1))|^2
\\&=\int_{\mathbb R^2}\dd y_2\dd y_3 \mathbb I_{\{|y_2|\leq1;|y_3|\geq1\}}|y_2|^{2\alpha}|\Phi^{w}_{s_1; s_2}(2y_2y_3)|^2
\\&\lesssim\|\Phi^w\|_{\mathcal{W}^{\rho,\gamma}_T} |s_2-s_1|^{2\gamma}\int_{|y_2|\leq1}|y_2|^{2\alpha}(\int_{|y_3|\geq1}\frac{1}{(1+|y_2y_3|)^{2\rho}}\dd y_3)\dd y_2
\\&\lesssim \|\Phi^w\|_{\mathcal{W}^{\rho,\gamma}_T}^2 |s_2-s_1|^{2\gamma}(\int_{|y_2|\leq1}\frac{1}{|y_2|^{1-2\alpha}}\dd y_2)(\int_{\mathbb R}\frac{1}{(1+|z_3|)^{2\rho}}\dd z_3)<+\infty
\end{split}
\end{equation*}
 when $\rho>1/2$, $\alpha>0$ and this give us the bound for $I_3$, we remark also in the case $\alpha=0$ the integral $I_3$ and $I_4$ are essentially the same by symmetry and can be bounded using the same argument. Now we will focus to bound the term $J_2$ for that we simply use the fact that  $|\Phi^{w}_{s_1; s_2}(a)|\leq|s_2-s_1|$ which is valid for all $a\in \mathbb R$. Then:
\begin{equation*}
\begin{split}
J_2^2&=\sup_{x_1}\int_{\mathbb R^2}\dd x_2\dd x_3 \mathbb I_{\{|x_2-x_1|\leq1;|x_3-x_1|\leq1\}}|x_2-x_1|^{2\alpha}|\Phi^{w}_{s_1; s_2}(2(x_2-x_1)(x_3-x_1))|^2
\\&=\int_{\mathbb R^2}\dd y_2\dd y_3 \mathbb I_{\{|y_2|\leq1;|y_3|\leq1\}}|y_2|^{2\alpha}|\Phi^{w}_{s_1; T}(2y_2y_3)|^2
\\&\lesssim |s_2-s_1|^2.
\end{split}
\end{equation*}
All these bounds give us estimates for $(I_l)$, $l\in\{1,2,3\}$, let us focus on the remaning integrals. To bound the integral $I_4$ we proceed in a different way, to simplify the notation let $\eta=2(x_2-x_1)(x_3-x_1)$ and then use the Cauchy-Schwarz inequality to get 
$$
\int_{\mathbb R}\dd x_2\mathbb I_{|x_2-x_1|\geq1}|x_2-x_1|^{\alpha}|\Phi_{s_1; T}(\eta)||\hat\psi_{2}(x_2)| |\hat\psi_{4}(x)|\leq\sup_{x_2}( \mathbb I_{|x_2-x_1|\geq1}|x_2-x_1|^{\alpha}|\Phi^w_{s_1; s_2}(\eta)|)\|\psi_{2}\|_{L^2(\mathbb R)}\|\psi_{4}\|_{L^2(\mathbb R)}
$$
now injecting this inequality in $I_4$ and using Cauchy-Schwarz and Young inequality we obtain that   
\begin{equation*}
\begin{split}
I_4&\leq(\int_{\mathbb R^2}\dd x_3\dd x_1\mathbb I_{|x_3-x_1|\leq1}\sup_{x_2}( \mathbb I_{|x_2-x_1|\geq1}|x_2-x_1|^{\alpha}|\Phi^{w}_{s_1; s_2}(\eta)|)|\hat\psi_{1}(x_1)||\hat\psi_{3}(x_3)|)\|\psi_{2}\|_{L^2(\mathbb R)}\|\psi_{4}\|_{L^2(\mathbb R)}
\\&=(\int_{\mathbb R}|\hat\psi_{s_1}(x_1)|(\int_{\mathbb R}\mathbb I_{|x_3-x_1|\leq1}\sup_{x_2}( \mathbb I_{|x_2-x_1|\geq1}|x_2-x_1|^{\alpha}|\Phi^{w}_{s_1; s_2}(\eta)|)|\hat\psi_{3}(x_3)|\dd x_3)\dd x_1)\|\psi_{2}\|_{L^2(\mathbb R)}\|\psi_{4}\|_{L^2(\mathbb R)}
\\&\leq \left(\int_{\mathbb R}\left|\int_{\mathbb R}\mathbb I_{|x_3-x_1|\leq1}\sup_{x_2}(\mathbb I_{|x_2-x_1|\geq1}|x_2-x_1|^{\alpha}|\Phi^{w}_{s_1; s_2}(\eta)|)|\hat\psi_{3}(x_3)|\dd x_3\right|^2\dd x_1\right)^{1/2}\|\psi_{2}\|_{L^2(\mathbb R)}\|\psi_{4}\|_{L^2(\mathbb R)}\|\psi_1\|_{L^2(\mathbb R)}
\\&\leq\int_{|y_3|\leq1}\sup_{|y_2|\geq1}(|y_2|^{\alpha}|\Phi^{w}_{s_1; s_2}(2y_2y_3)|)\dd y_3\Pi_{i=1}^{4}\|\psi_{i}\|_{L^2(\mathbb R)}
\\&\lesssim\|\Phi^w\|_{\mathcal{W}^{\rho,\gamma}_T} |T-s_1|^{\gamma}\Pi_{i=1}^{4}\|\psi_{i}\|_{L^2(\mathbb R)}\sup_{z_2}\left(|z_2|^{\alpha}(1+|z_2|)^{-\rho}\right)\int_{|y_3|\leq1}|y|^{-\alpha}\dd y_3<+\infty
\end{split}
\end{equation*}
when $\alpha<1$ and $\rho>\alpha$. As was noted previously this gives us also a bound for $I_3$ when $\alpha=0$. 
Now to treat the case $\alpha=1$ we proceed as in~\cite{Debussche2011363}. Indeed after change of variable we can rewrite our integral as:
$$
I(1)=\int_{\mathbb R}|x|\left(\int_{\mathbb R}|\hat\psi_{1}(y_1)||\hat\psi_{2}(x-y_1)|\left(\int_{\mathbb R}|\hat\psi_{3}(y_2)||\hat\psi_{4}(x-y_2)||\Phi^{w}_{s_1; s_2}(2x(y_2-y_1))|\dd y_2\right)\dd y_1\right)\dd x
$$ 
and then by Cauchy-Schwarz and Young inequality we have
\begin{equation*}
\begin{split}
I(1)&\leq(\sup_x|x|\int_{\mathbb R}|\Phi^{w}_{s_1; T}(2xz)|\dd z)\Pi_{i=1}^{4}\|\psi_{i}\|_{L^2(\mathbb R)}
\lesssim \|\Phi^w\|_{\mathcal{W}^{\rho,\gamma}_T}|s_2-s_1|^{\gamma}(\int_{\mathbb R}(1+|z|)^{-\rho}\dd z)\Pi_{i=1}^{4}\|\psi_{i}\|_{L^2(\mathbb R)}.
\end{split}
\end{equation*}
The r.h.s is finite if $\rho>1$ and this concludes the proof.
\end{proof}

\subsection{Global existence for the modulated cubic NLS equation}
\label{sec:global-NLS}

Here we obtain global solution of positive regularity for the modulated cubic NLS equation: we take $\cN(\varphi) = i \varphi |\varphi|^2$ and we work indifferently in $\RR$ or $\TT$. 

\begin{theorem}\label{prop:Global-Schr}
Fix $\alpha \ge 0$,  $\phi\in H^{\alpha}$ and $T>0$. Then there exist $v\in C^{1/2}([0,T],H^\alpha)$ such that the following equation holds 
$$
v_t=\phi+\int_{0}^{t}X_{\dd\sigma}(v_{\sigma})
$$
for all $t\in[0,T]$. 
\end{theorem}
In order to prove this result we will need to show that the operator $X$ corresponding to NLS obeys the $L^2$ conservation law. 
\begin{lemma}
\label{lemma:conservation}
 We have that for any $\phi\in L^2$ and any $0\le s \le t$: 
$$
\langle \phi, X_{s;t}(\phi,\phi,\phi)\rangle \in \RR
$$
and there exists a constant $C_R$ such that for all $\phi\in L^2$ with $\|\phi\|_{L^2}\le R$ we have
$$
|\|\phi+X_{s;t}(\phi,\phi,\phi)\|_{L^2}-\|\phi\|_{L^2}|\le C_R |t-s|^{2\gamma}.
$$
\end{lemma}

\begin{proof}
We start observing that for smooth $\phi$:
$$
\langle \phi, \dot X_{s}(\phi,\phi,\phi)\rangle =2|\phi|^4_{L^2} +\langle \phi, U^{w}_{-s}( |U_s \phi|^2 U^{w}_s \phi)\rangle 
= 2|\phi|_2^2+\langle U^{w}_s \phi,  |U^{w}_s \phi|^2 U^{w}_s \phi \rangle\in \RR
$$
Integrating in $s$ and extending to arbitrary $\phi\in L^2$ we get the claim. Then if $\phi\in L^2$  we have
$$
\|\phi+X_{s;t}(\phi,\phi,\phi)\|^2_{L^2}=\|\phi\|^2_{L^2}+\|X_{s;t}(\phi,\phi,\phi)\|^2_{L^2}
$$
so
$$
|\|\phi+X_{s;t}(\phi,\phi,\phi)\|_{L^2}-\|\phi\|_{L^2}|\le \frac{\|X_{s; t}(\phi,\phi,\phi)\|^2_{L^2}}{\|\phi\|_{L^2}} \lesssim |t-s|^{2\gamma} \|\phi\|_{L^2}^5 .
$$
\end{proof}

\begin{proof}[Proof of Theorem~\ref{prop:Global-Schr}]
The existence of the global solution in $L^2$ is given by the conservation law $||v_t||_{L^2}=||\phi||_{L^2}$ which for local solutions holds thanks to the previous lemma and to the Lemma~\eqref{lemma:global}. Then by patching local solutions   we obtain a global solution. Let us now turn to the proof of existence of global solution in $H^\alpha$ for any regularity $\alpha >0$. We will describe the details of the computations in the periodic setting. The non-periodic setting allows for a similar but simpler proof since resonances will not play any role. Recall the decomposition  $X=X^1+X^2$ introduced in Proposition~\ref{th:reg-1d} and decompose further $X^2$ as 
$$
X^2=X^{21}+X^{22}+X^{23}
$$
with 
$$
\cF X_{t}^{2j}(\psi_1,\psi_2,\psi_3)(k)=\sum_{(k_1,k_2,k_3)\in D^k_j}\hat\psi_1(k_1)^{\star}\hat\psi_2(k_2)\hat\psi_3(k_3)\Phi^w_{t}((k-k_2)(k-k_3)) 
$$
 for $j\in\{1,2,3\}$ and where $D^k_1=\{-k_1+k_2+k_3=k,k_2\ne k,k_3\ne k\}\cap\{|k_1|\geq|k|/3\}$, $D^k_2=\{-k_1+k_2+k_3=k,k_2\ne k,k_3\ne k\}\cap\{|k_1|<|k|/3,|k_2|\geq|k|/3\}$ and $D^k_3=\{-k_1+k_2+k_3=k,k_2\ne k,k_3\ne k\}\cap\{|k_1|<|k|/3,|k_2|<|k|/3,|k_3|\geq|k|/3\}$. Using Cauchy-Schwarz inequality we have the following bound  
\begin{equation}
\begin{split}
||X^{2j}_{s; t}(\psi_1,\psi_2,&\psi_3)||^2_{H^{\beta+\eps}}\leq||\psi_j||_{H^{\beta+\eps}}\Pi_{i\ne j}||\psi_i||_{H^\beta} \times
\\ &\times \sup_k|k|^{2\beta+2\eps} \sum_{D_j^k}|k_j|^{-2\beta-2\eps}\left(\Pi_{i\ne j}|k_i|^{-2\beta}\right) |\Phi^w_{s; t}(2(k-k_2)(k-k_3))|^2
\end{split}
\end{equation}
 for $\beta,\eps\geq0$ then using the fact that $|k|\lesssim|k_j|$ on $D_j^k$ and using the $\rho$-irregularity of $w$ we obtain 
$$
||X^{2j}_{s; t}(\psi_1,\psi_2,\psi_3)||^2_{H^{\beta+\eps}}\lesssim_{\alpha,\eps}||\Phi^w||_{\mathcal W^{\rho,\gamma}_T}|t-s|^{\gamma}||\psi_j||_{H^{\beta+\eps}}\Pi_{i\ne j}||\psi_i||_{H^\beta}\left(\sum_{l\ne0}|l|^{-2\rho}\right)^2<+\infty
$$
when $\rho>1/2$ and then we have that for all $T>0$ there exist $\gamma>1/2$ such that $X^{21}$ belongs to $ C^{\gamma}\left([0,T],\mathcal L^3(H^{\beta+\eps}\times H^{\beta}\times H^\beta,H^{\beta+\eps})\right)$ for all $\beta,\eps\geq0$. Of course the same statement holds for the other operators. Below we will use this result for $\beta=0$ and $\eps=\alpha$.  Now let us define a norm on $C^{1/2}([0,T],H^\alpha)$ by $||\psi||_{\alpha}=||\psi||_{ C^{1/2}([0,T],H^\alpha)}+||\psi||_{ C^0([0,T],H^\alpha)}$  for $\alpha\geq 0$ and the map $\Gamma$ by: 
$$
\Gamma(\psi):=\phi+\int_0^t X_{\dd\sigma}(\psi_\sigma)
$$
for $\psi\in  C^{1/2}([0,T],H^\alpha)$. By a simple computation we see that 
$$
||\Gamma(\psi)||_{0}\lesssim_{\gamma,w}||\phi||_{L^2}+T^{\gamma-1/2}||\psi||_{0}^3.
$$
Now if $0<T\leq T_1$ is sufficiently small then the equation $r=||\phi||_{L^2}+T^{\gamma-1/2}r^3$ admits a positive solution $r^{\star}>0$ and  the closed ball $B_{r^{\star}}=:\{\psi\in C^{1/2}([0,T],L^2);||\psi||_0\leq r_T^\star\}$ is invariant by $\Gamma$. Moreover we have that 
$$
||\Gamma(\psi_1)-\Gamma(\psi_2)||_0\lesssim_{\gamma,w} T^{\gamma-1/2}||\psi_1-\psi_2||_0(1+(r^{\star})^2)
$$
and then if $T\leq T_2\leq T_1$ sufficiently small,  $\Gamma$ is a strict contraction on $B_{r^{\star}}$ which admits a unique fixed point $v$. Let $\Gamma_{B_{r^{\star}}}$ the restriction of $\Gamma$ on $B_{r^{\star}}$ and use the fact that
$$
\Gamma(\psi)_t=\phi+2\int_0^t\psi_\sigma||\psi_\sigma||_{L^2}^2\dd\sigma+\sum_{j\in\{1,2,3\}}\int_{0}^{t}X^{2j}_{\dd\sigma}(\psi_\sigma)
$$
and the regularity of $X^{2j}$ to deduce that
$$
||\Gamma_{B_{r^{\star}}}(\psi)||_{\alpha}\lesssim||\phi||_{H^\alpha}+T^{\gamma-1/2}(r^{\star})^2||\psi||_{\alpha}.
$$
Then $B(0,R):=\{\psi\in C([0,T],H^\alpha);||\psi||\leq R\}$ is invariant by $\Gamma_{B_{r^{\star}}}$ for $T^{\star}=T^{\star}(\|\phi\|_{L^2})$ small enough depending only on $\|\phi\|_{L^2}$ since $r^\star=r^\star(\|\phi\|_{L^2})$. Given that the ball $B(0,R)$ is closed in $C^{1/2}([0,T],L^2)$  we have that $v\in C^{1/2}([0,T^{\star}],H^{\alpha})$ and using the conservation law in $L^2$ we can repeat the argument on $[T^\star,2T^{\star}]$ and  recursively  obtain that $v\in C^{1/2}([0,T],H^\alpha)$.
\end{proof}
\subsection{Convergence of regularized models}
We study here the convergence of approximations given by standard PDEs to the solution of the Young equations. Consider the following regularized problem 
\begin{equation}\label{eq:Cauchy-reg}
\left\{
\begin{aligned}
\partial_t \varphi_t& = A \varphi_t\partial_tn_t + \Pi_{L}\cN(\Pi_{L}\varphi_t),\qquad t\ge 0,
\\&\varphi(0,x)=\Pi_{L} \phi(x)\in C^{\infty}(\mathbb T)
\end{aligned}
\right.
\end{equation}
with $n$ being a differentiable function, $\phi\in H^\alpha(\mathbb T)$ for $\alpha>0$, $A=i\partial^2_x$ and $\mathcal N(\phi)=i|\phi^2|\phi$. This Cauchy problem is equivalent to the mild formulation 
\begin{equation}\label{eq:mild-phi}
\varphi_t=U^{n}_t\Pi_{L}\phi+\int_{0}^{t}U^{n}_t(U^{n}_{s})^{-1}\Pi_{L}\cN(\Pi_{L}\varphi_s)\dd s,\qquad t\ge 0
\end{equation}
or equivalently
\begin{equation}\label{eq:mild-reg}
\psi_t=\Pi_{L}\phi+\int_{0}^{t}(U^{n}_{s})^{-1}\Pi_{L}\cN(\Pi_{L}U^{n}_s\psi_s)\dd s,\qquad t\ge 0,
\end{equation}
with $U^{n}_t=e^{An_t}$ and $\psi_t=(U^{n})^{-1}_t\varphi_t$. Now we can check easily that  the modulated operator $X^{n,L}$ associated to the equation~\eqref{eq:mild-reg} is well defined and satisfies
$$
||X_{s; t}^{n,L}||_{\mathcal L^2(H_{\alpha_1},H_{\alpha_2})}\lesssim_{n,L}|t-s|
$$
for all $\alpha_1,\alpha_2\in \mathbb R$. By a fixed point argument we obtain the existence of a unique Young local solution $\varphi^{n,L}\in C([0,T^{\star}],L^2)$ such that $ \psi^{n,L}_t=(U_t^{n})^{-1}\varphi^{n,L}_t\in C^{1}([0,T^{\star}],L^2)$ moreover we have that $\psi^{n,L}\in\cap_{\beta\geq0} C^{1}([0,T^{\star}],H_{\beta})$ and then clearly 
$$
\partial_t \varphi_t = A \varphi_t\partial_tn_t + \Pi_{L}\cN(\Pi_{L}\varphi_t)
$$
in the weak sense. To obtain a global solution is sufficient to remark that for all $v\in L^2$
\begin{equation*}
\begin{split}
\langle v,X^{n,L}_{s; t}(v)\rangle&=2(t-s)\|v\|^4_{L^2}+\int_{s}^{t}\dd\sigma\int_{\mathbb T}U^{n}_{\sigma} v\Pi_{L}(|\Pi_{L}U^n_\sigma v|^2(U_\sigma \Pi_{L}v)^{\star})
\\&=2(t-s)\|v\|^4_{L^2}+\int_{s}^{t}\dd\sigma\int_{\mathbb T}|U_\sigma^n\Pi_{L}v|^4\in\mathbb R .
\end{split}
\end{equation*}
This implies
\begin{equation*}
\begin{split}
||\psi^{n,L}_t||^2_{L^2}&=||\psi^{n,L}_s||^2_{L^2}+||\psi^{n,L}_t-\psi^{n,L}_s||_{L^2}^2+2\mathcal Re(\langle\psi^{n,L}_s,iX^{n,L}_{st}(\psi^{n,L}_s,\psi^{n,L}_s,\psi^{n,L}))\rangle+R_{s, t}
\\&=||\psi^{n,L}_s||^2_{L^2}+||\psi^{n,L}_t-\psi^{n,L}_s||_{L^2}^2+R_{s, t}
\end{split}
\end{equation*}
for all $s,t\in[0,T^{\star}]$ with $|R_{s, t}|\lesssim|t-s|^2$. From this follows $|||\psi^{n,L}_t||^2_{L^2}-||\psi^{n,L}_s||^2_{L^2}|\lesssim |t-s|^2$, that is  $||\psi^{n,L}_t||_{L^2}=||\Pi_{L}\phi||_{L^2}$ for all $t\ge 0$. Using this conservation law we can extend our local solution to a global one in $L^2$ and using the same proof of Theorem~\eqref{prop:Global-Schr} we construct a global solution in $H^\alpha$. The mild eq.~\eqref{eq:mild-reg} has a meaning even when $n$ is only a continuous function. Let $R>0$, $T>0$ and assume that $\sup_{\sigma\in[0,T]}|n_\sigma|\leq R$ then we obtain  
$$
||\psi^{n,L}||_{C^{1-\eps}([0,T],H^\alpha)}\lesssim_{L}T_1^{\eps}||X^{n,L}||_{C^1([0,T],\mathcal L^2H^\alpha)}(||\psi^{n,L}||_{C^{1-\eps}([0,T_1],H^\alpha)}+||\Pi_{L}\phi||_{H^\alpha})^2
$$
for all $T_1<\min(1,T)$, using the fact that $||X^{n,L}||_{C^1([0,T],\mathcal L^2H^\alpha)}\lesssim_{L}\sup_{\sigma\in[0,T]}|n_{\sigma}|\lesssim_{L}R$ and taking $T_1=T_1(||\Pi_{L}\phi||_{H^\alpha})$ small enough we can see that $||\psi^{n,L}||_{C^{1-\eps}([0,T_1],H^\alpha)}\lesssim_{L}R$. Finally iterating these results on contiguous small subintervals of $[0,T]$ gives us that $||\psi^{n,L}||_{C^{1-\eps}([0,T],H^\alpha)}\lesssim_{L}R$.  By a similar argument we obtain easily 
$
||\psi^{n^2,L}-\psi^{n^1,L}||_{C^{1-\eps}([0,T],H^\alpha)}\lesssim_{L,R}\sup_{\sigma\in[0,T]}|n^1_{\sigma}-n^2_{\sigma}|
$
for all $n_1,n_2\in C([0,T])$ such that $\sup_{\sigma\in[0,T]}|n^i_{\sigma}|\leq R$ for $i=1,2$ where $\psi^{n^1,L}$,$\psi^{n^2,L}$ are respectively the global solution of the eq.~\eqref{eq:mild-reg} associated to the dispersion $n^1$ and $n^2$. Now let $w^{N}$ a regularization of the continuous $\rho$--irregular function $w$   and assume that  
$
\sup_{\sigma\in[0,T]}|w^N_{\sigma}-w_{\sigma}|\to_{N\to+\infty}0
$
for all $T>0$. Then the solutions $(\varphi^{N,L})_{N\in\mathbb N}$ of the regularized problem~\eqref{eq:Cauchy-reg} with  dispersion $w^N$ converge in $C([0,T],H^\alpha)$ to $\varphi^{L}$ which is the solution of the mild equation~\eqref{eq:mild-phi} with dispersion $w$:
\begin{equation}\label{eq:reg-Galerkin}
\varphi^{L}_t=U^{w}\Pi_{L}\phi+\int_{0}^{t}(U^{w}_t)(U^{w}_s)^{-1}\Pi_{L}\mathcal N(\Pi_{L}\varphi_s)\dd s .
\end{equation}
Finally we have
\begin{theorem}\label{th:conv-Gal}
Let $\rho>1/2$, $\alpha>0$, $T>0$ and $\varphi^{L}$, $\varphi$ respectively the solution of the mild eq.~\eqref{eq:reg-Galerkin} on $[0,T]$ and the modulated cubic NLS equation then
$$
||\psi^{L}-\psi||_{C^{1/2}([0,T],H^\alpha)}\to^{L\to+\infty}0
$$ 
with $\psi^L_t=(U^{w}_t)^{-1}\varphi^{L}_t$ and $\psi_t=(U_t^{w})^{-1}\varphi_t$. 
\end{theorem}

\subsection{Cubic non linear Schr\"odinger equation on $\RR^2$}
\label{sec:2d-cubic-nls}
To obtain a local existence result for the modulated Schr\"odinger equation on $\RR^2$ we need to obtain regularity estimates for the appropriate  modulated operator $X$. Here $\dot X_s(\psi_1,\psi_2,\psi_3)$ is a trilinear operator with Fourier transform given by
\begin{equation*}
\cF \dot X_{s}(\psi_1, \psi_2, \psi_3) (\xi) = \int_{\substack{\xi_1,\xi_2,\xi_3 \in\RR^2\\ \xi=-\xi_1+\xi_2+\xi_3}} e^{i w_s (|\xi|^2+|\xi_1|^2-|\xi_2|^2-|\xi_3|^2)} \, \hat \psi_1(\xi_1)^* \hat \psi_2 (\xi_2) \hat \psi_3(\xi_3) \dd \xi_2 \dd \xi_3 .
\end{equation*}
 Note that  $|\xi|^2+|\xi_1|^2-|\xi_2|^2-|\xi_3|^2 =2 \langle \xi -\xi_2,\xi - \xi_3\rangle_{\RR^2}= \Xi$ under the condition that $\xi=-\xi_1+\xi_2+\xi_3$. Then $X$ has the expression
\begin{equation*}
\cF  X_{t}(\psi_1, \psi_2, \psi_3) (\xi) = \int_{\substack{\xi_1,\xi_2,\xi_3 \in\RR^d\\ \xi=-\xi_1+\xi_2+\xi_3}} \Phi^{w}_{t}(\Xi)\, \hat \psi_1(\xi_1)^* \hat \psi_2 (\xi_2) \hat \psi_3(\xi_3) \dd \xi_2 \dd \xi_3 .
\end{equation*}
Using the $(\rho,\gamma)$-iregularity of $w$ we can easily obtain that 

\begin{equation*}
|\langle \psi, X_{s;t}(\psi_1,\psi_2,\psi_3)\rangle| \le
\int_{\substack{\xi,\xi_1,\xi_2,\xi_3 \in\RR^2\\ \xi=-\xi_1+\xi_2+\xi_3}} 
|\Phi^{w}_{s;t}(\Xi)|
 \,|\hat \psi(\xi)|  |\hat \psi_1(\xi_1)| |\hat \psi_2 (\xi_2)| | \hat \psi_3(\xi_3)| \dd \xi_1 \dd \xi_2 \dd \xi_3
\end{equation*}
\begin{equation*}
\le J^{1/2}  \|\Phi^w\|_{\mathcal{W}^{\rho,\gamma}_T} |t-s|^\gamma\|\psi\|_{-\alpha} \|\psi_1\|_\alpha \|\psi_2\|_\alpha \|\psi_3\|_\alpha
\end{equation*}
with
$$
 J=\sup_{\xi\in\RR^2} \int_{\substack{\xi_1,\xi_2,\xi_3 \in\RR^2\\ \xi=-\xi_1+\xi_2+\xi_3}}  (1+2|\langle \xi -\xi_2,\xi - \xi_3\rangle|)^{-2\rho} (1+|\xi|^2)^{\alpha} \prod_{i=1,2,3}(1+|\xi_i|^2)^{-\alpha} \dd \xi_2 \dd \xi_3 
$$

\begin{lemma}
The quantity $J$ is finite when $\alpha > 1/2$ and $\rho > 1/2$.
\end{lemma}
\begin{proof}
Inserting the  estimate
$
 (1+|\xi|^2)^{\alpha} \lesssim  \sum_{i=1}^3 (1+|\xi_i|^2)^{\alpha}
$
we obtain that $J=J_1 + J_2$ where
$$
 J_1=\sup_{\xi\in\RR^2} \int_{\substack{\xi_1,\xi_2,\xi_3 \in\RR^2\\ \xi=-\xi_1+\xi_2+\xi_3}}  (1+2|\langle \xi -\xi_2,\xi - \xi_3\rangle|)^{-2\rho} (1+|\xi_2|^2)^{-\alpha} (1+|\xi_3|^2)^{-\alpha} \dd \xi_2 \dd \xi_3 
$$
and
$$
 J_2=\sup_{\xi\in\RR^2} \int_{\substack{\xi_1,\xi_2,\xi_3 \in\RR^2\\ \xi=-\xi_1+\xi_2+\xi_3}}  (1+2|\langle \xi -\xi_2,\xi - \xi_3\rangle|)^{-2\rho} (1+|\xi_2|^2)^{-\alpha} (1+|\xi_1|^2)^{-\alpha} \dd \xi_2 \dd \xi_3 
$$
Let us consider first the $J_1$ contribution. Let $q_i = \xi-\xi_i$, $i=2,3$
$$
 J_1=\sup_{\xi\in\RR^2} \int_{\RR^2}  \frac{\dd q_2}{ (1+|\xi + q_2|^2)^{\alpha}}  \int_{\RR^2} \frac{\dd q_3} { (1+2|\langle q_2,q_3\rangle|)^{2\rho} (1+|\xi + q_3|^2)^{\alpha}}
$$
Write $q_3^\bot,q_3^\|\in\RR$ for the perpendicular and parallel components of $q_3\in\RR^2$ with respect to $q_2$ and similarly for $\xi$ and bound
$$
 J_1\le\sup_{\xi\in\RR^2} \int_{\RR^2}  \frac{\dd q_2}{ (1+|\xi + q_2|^2)^{\alpha}}  \int_{\RR} \frac{\dd q^\bot_3}{(1+|\xi^\bot + q^\bot_3|^2)^{\alpha}} \int_{\RR} \frac{\dd q^\|_3} { (1+2|q_2||q^\|_3|)^{2\rho} }
$$
$$
 = \sup_{\xi\in\RR^2} \int_{\RR^2}  \frac{\dd q_2}{ (1+|\xi + q_2|^2)^{\alpha}}  \int_{\RR} \frac{\dd q^\bot_3}{(1+|q^\bot_3|^2)^{\alpha}} \int_{\RR} \frac{\dd q^\|_3} { (1+2|q_2||q^\|_3|)^{2\rho} }
$$
now note that for $\alpha > 1/2$ and $\rho > 1/2$ we have
$$
\int_{\RR} \frac{\dd q^\bot_3}{(1+| q^\bot_3|^2)^{\alpha}} \int_{\RR} \frac{\dd q^\|_3} { (1+2|q_2||q^\|_3|)^{2\rho} } \lesssim |q_2|^{-1}
$$
so that
$$
 J_1\lesssim\sup_{\xi\in\RR^2} \int_{\RR^2}  \frac{\dd q_2}{ (1+|\xi + q_2|^2)^{\alpha} |q_2|} <+\infty
$$
for $\alpha>1/2$.  To estimate the $J_2$ integral we rewrite it as 
$$
 J_2=\sup_{\xi\in\RR^2} \int_{\RR^2} \int_{\RR^2} (1+2|\langle \xi_1 -\xi_2,\xi - \xi_2\rangle|)^{-2\rho} (1+|\xi_2|^2)^{-\alpha} (1+|\xi_1|^2)^{-\alpha} \dd \xi_2 \dd \xi_1 
$$
where we used that $\xi-\xi_3 = \xi_2-\xi_1$. By writing $q_1 = \xi_1 -\xi_2$ and $q_2 = \xi-\xi_2$ we get
$$
 J_2=\sup_{\xi\in\RR^2} \int_{\RR^2} \frac{\dd q_2}{(1+|\xi-q_2|^2)^{\alpha}} \int_{\RR^2}\frac{\dd q_1}{ (1+2|\langle q_1,q_2\rangle|)^{2\rho}  (1+|q_1+\xi-q_2|^2)^{\alpha} }
$$
Write $q_1^\bot,q_1^\|$ for the perpendicular and parallel components of $q_1$ with respect to $q_2$ to get the estimate
$$
 J_2\le \sup_{\xi\in\RR^2} \int_{\RR^2} \frac{\dd q_2}{(1+|\xi-q_2|^2)^{\alpha}} \int_{\RR^2}\frac{\dd q^\bot_1 \dd q^\|_1}{ (1+2| q_1^\|| |q_2|)^{2\rho}  (1+|q^\bot_1+\xi^\bot|^2)^{\alpha} }
$$
again the condition $\alpha > 1/2$ allows to bound this last quantity as
$$
\lesssim \sup_{\xi\in\RR^2} \int_{\RR^2} \frac{\dd q_2}{(1+|\xi-q_2|^2)^{\alpha}} \int_{\RR}\frac{ \dd q^\|_1}{ (1+2| q_1^\|| |q_2|)^{2\rho} }
$$
and $\rho > 1/2$ subsequently by
$$
\lesssim \sup_{\xi\in\RR^2} \int_{\RR^2} \frac{\dd q_2}{(1+|\xi-q_2|^2)^{\alpha}|q_2|}
$$
which is finite when $\alpha > 1/2$.
\end{proof}

\begin{theorem}
\label{th:2d-reg}
For all $\rho>1/2$ there exists $\gamma>1/2$ such that for all $T>0$ the operator $X$ belongs to $\cC^\gamma([0,T],\LL_3 H^\alpha(\RR^2))$ for all $\alpha>1/2$. 
\end{theorem}

\subsection{The derivative NLS equation}
\label{sec:dnls}

Here we will focus on the modulated Derivative Nonlinear Schr\"odinger equation (ie:  $A=i\partial^2$ and $\mathcal N(u)=i (-i\partial)^{\theta}(|u|^2-||u||_2^2)u$ for $\theta>0$. Now the Fourier transform of the operator associated to this equation is given by
\begin{equation}\label{eq:dNLS}
\hat X_{t}(\psi_1,\psi_2,\psi_3)=i(k)^{\theta}\sum_{\star} \psi_1(k_1)^{*}\psi_2(k_2)\psi_3(k_3)\Phi^{w}_{t}(2(k-k_2)(k-k_3))
\end{equation}
where the star under the sum means that we have $-k_1+k_2+k_3=k$ and $k_2\ne k$,$k_3\ne k$
, $k_1k_2k_3\ne0$. Standard application of Cauchy-Schwartz gives
$$
||X_{s;t}||_{\cL_3 H^{\alpha}}^2\leq\sup_{k\ne0}|k|^{2\alpha+2\theta}\sum_{\star}|k_1k_2k_3|^{-2\alpha}|\Phi^{w}_{s;t}(2(k-k_2)(k-k_3))|^2
$$
then using the fact that $w$ is $(\gamma,\rho)$ irregular we obtain  
\begin{equation*}
||X_{s; t}||^2_{\cL_3 H^{\alpha}}\lesssim \|\Phi^w\|_{\mathcal{W}^{\rho,\gamma}_T} |t-s|^{\gamma}\sup_k|k|^{2+2\alpha}\sum_{\star} |k_1k_2k_3|^{-2\alpha}|k-k_2|^{-2\rho}|k-k_3|^{-2\rho} 
\end{equation*}
with $s,t\in[0,T]$, then is sufficient to prove that  
$$
I=\sup_k|k|^{2\alpha+2\theta}\sum_{*}|k_1k_2k_3|^{-2\alpha}|k-k_3|^{-2\rho}|k-k_2|^{-2\rho}<+\infty
$$
for that we will need the following lemma.
\begin{lemma}\label{lemma:sum-bound}
For $\rho>\max(1/2,\theta/2)$  and $\alpha\geq\frac{1}{2}\theta$  then the following inequality holds: 
\begin{equation*}
\sum_{l\ne0,k}|l|^{-2\alpha}|k-l|^{-2\rho}\lesssim|k|^{-\theta}
\end{equation*}
\end{lemma}
\begin{proof}
\begin{equation*}
\begin{split}
\sum_{l\ne k,0}|l|^{-2\alpha}|k-l|^{-2\rho}&=\sum_{l\ne0,k;|l-k|\leq|l|}|l|^{-2\alpha}|k-l|^{-2\rho}+\sum_{l\ne 0,k;|k-l|\geq|l|}|l|^{-2\alpha}|k-l|^{-2\rho}
\\&\leq|k|^{-\theta}\sum_{l\ne0}\frac{1}{|l|^{2\rho}}+|k|^{-\theta}\sum_{l}\frac{1}{|l|^{2\alpha+2\rho-\theta}}
\lesssim_{\theta,\alpha,\rho}\frac{1}{|k|^{\theta}} .
\end{split}
\end{equation*}
\end{proof}
\begin{lemma}
Let $\rho>\max(\theta,1/2)$ and $\alpha\geq\frac{1}{2}\theta$  then $I<+\infty$.
\end{lemma}
\begin{proof}
If we use the fact that $|k|^{2\alpha}\lesssim|k_1|^{2\alpha}+|k_2|^{2\alpha}+|k_3|^{2\alpha}$ we obtain 
\begin{equation*}
\begin{split}
I&\lesssim
\sup_k|k|^{2\theta}\left(\sum_{*}|k_2k_3|^{2\alpha}|k-k_2|^{-2\rho}|k-k_3|^{-2\rho}+\sum_{*}|k_1k_2|^{-2\alpha}|k-k_2|^{-2\rho}|k_2-k_1|^{-2\rho}\right) 
\\&\lesssim\sup_k|k|^{2\theta}\left(\sum_{k_2\ne0,k}|k_2|^{-2\alpha}|k-k_2|^{-2\rho}\right)^{2}+\sup_k|k|^{2\theta}\sum_{*}|k_1k_2|^{-2\alpha}|k-k_2|^{-2\rho}|k_2-k_1|^{-2\rho} 
\\&=I_1+I_2
\end{split}
\end{equation*}
Now by the Lemma~\ref{lemma:sum-bound} we have
$$
I_1=\sup_k|k|^{2\theta}(\sum_{k_2\ne0,k}|k_2|^{-2\alpha}|k-k_2|^{-2\rho})^{2}<+\infty
$$
for $\rho>\max(\theta,1/2)$, $\alpha>\frac{1}{2}\theta$.
It remains to treat the second term which requires a bit more work:
\begin{equation*}
\begin{split}
I_{2}&=\sup_k|k|^{2\theta}\sum_{k_1,k_2}|k_2k_1|^{-2\alpha}|k_2-k|^{-2\rho}|k_2-k_1|^{-2\rho}
\\&=\sup_k|k|^{2\theta}\sum_{k_2}|k_2|^{-2\alpha}|k-k_2|^{-2\rho}\sum_{k_1}|k_1|^{-2\alpha}|k_2-k_1|^{-2\rho}\lesssim\sup_k|k|^{2\theta}\sum_{k_2}|k_2|^{-2\alpha-\theta}|k-k_2|^{-2\rho}
\\&\lesssim\sup_k|k|^{2\theta}(\sum_{k_2;|k-k_2|\leq|k_2|}|k_2|^{-2\alpha-\theta}|k-k_2|^{-2\rho}+\sum_{k_2;|k-k_2|\geq|k_2|}|k_2|^{-2\alpha-\theta}|k-k_2|^{-2\rho})
\\&\lesssim\sup_k(|k|^{\theta-2\alpha})+\sup_k\sum_{k_2;|k-k_2|\geq|k_2|}|k_2|^{-2\alpha-\theta}|k-k_2|^{-2\rho+2\theta}
\\&\lesssim \sup_k(|k|^{\theta-2\alpha})+\sum_{k_2\ne0}|k_2|^{-2\alpha-2\rho+\theta}<+\infty
\end{split}
\end{equation*}
if  $\rho>\max(1/2,\theta) $, $\alpha\geq\frac{1}{2}\theta$.
\end{proof}
\begin{theorem}\label{th:rg-DNLS}
Let $\theta>0$, $\rho>\max(1/2,\theta)$ then there exist $\gamma>1/2$ such that  the operator $X$ defined the formula~\eqref{eq:dNLS} satisfies $X\in  C^{\gamma}([0,T],\LL_3 (H^{\alpha}))$ for all $T>0$ and $\alpha\geq\frac{1}{2}\theta$.
\end{theorem}
\begin{remark}
Taking $\theta=1$ and $\rho>1$ we obtain  the existence and uniqueness of a local solution for the modulated dNLS equation as stated in Theorem~\ref{th:nse-others} (provided as usual that $w$ is $\rho$--irregular).
\end{remark}

\section{Modulated Strichartz estimates and  NLS on $\RR$}
\label{sec:strich-NLS}
In this section we study the modulated Schr\"odinger equation with general sub--critical non linearity, i.e.  $A=i\partial^2$ and $\mathcal N(\phi)= i |\phi|^\mu \phi$ where $\mu$ can be any real number in $[0,4]$. In the case of Brownian modulation this equation has already been studied in~\cite{Debussche2011363}  using a stochastic Strichartz type estimate for $\mu \le 4$. The goal of this section is to remark that their result generalizes easily to an arbitrary $\rho$-irregular path. In order to prove Theorem~\ref{theorem:strichartz-w} we will follow their strategy and obtain a preliminary estimate which involves computations similar to those used in the study of the $X$ operator for the cubic NLS in Lemma~\ref{th:reg-1d}.
\begin{theorem}\label{proposition: Strich}
Let $\alpha\in[0,1]$ and $\rho>0$ such that $0\leq\alpha<1$ and $\rho>\alpha+1/2$ or $\alpha=1$ and $\rho>1$ then there exists $\gamma>1/2$ such that we have 
$$
\int_{0}^{T}\dd t\left|\left|D^{\frac{\alpha}{2}}\left|\int_{0}^{t}U_t^{w}(U_s^{w})^{-1}\psi_s\dd s\right|^2\right|\right|^2_{L^{2}(\mathbb R)}\lesssim ||\Phi^w||_{\mathcal W_{\rho,T}^{\gamma}}T^{\gamma}\left|\left|\psi\right|\right|^4_{L^1([0,T],L^2(\mathbb R))}
$$
for every $\psi\in L^1([0,T],L^2(\mathbb R))$ and for all $T>0$. 
\end{theorem} 
\begin{proof}
By going to Fourier variables we can see that :
\begin{equation*}
\begin{split}
\int_{0}^{T}\dd t\left|\left|D^{\frac{\alpha}{2}}\left|\int_{0}^{t}U_t^{w}(U_s^{w})^{-1}\psi_s\dd s\right|^2\right|\right|^2_{L^{2}(\mathbb R)}= & \int_{0}^{T}\dd t\int_{[0,t]^4}\dd s_1\dd s_2\dd s_3\dd s_4\int_{\mathbb R^3}\dd x_1\dd x_2\dd x_3 \times
\\& \times (e^{-i\phi}|x_2-x_1|^{\alpha}|\hat\psi_{s_1}(x_1)||\hat\psi^*_{s_2}(x_2)||(\hat\psi_{s_3})^*(x_3)||\hat\psi^*_{s_4}(x_4)|)
\end{split}
\end{equation*}
where $x_4=-x_1+x_2+x_3$, $\phi=x_1^2(w_t-w_{s_1})-x_2^2(w_t-w_{s_2})-x_3^2(w_t-w_{s_3})+x_4^2(w_t-w_{s_4})$. Let us split the integral over $(s_1,s_2,s_3,s_4)$ in four region where $s_i=\max(s_1,s_2,s_3,s_4)$ for $i=1,...,4$. Consider for example the first region where $s_1>s_2,s_3,s_4$. Using Fubini we can see that this integral is given by  
$$
\mathcal I=\int_{0}^{T}\dd s_1\int_{[0,s_1]^{3}}\int_{\mathbb R^3}\dd x_1\dd x_2\dd x_3(\int_{s_1}^{T}e^{-i\phi}\dd t)|x_2-x_1|^{\alpha}|\hat\psi_{s_1}(x_1)||\hat\psi^*_{s_2}(x_2)||(\hat\psi_{s_3})^*(x_3)||(\hat\psi^*_{s_4})^{*}(x_4)|
$$
and 
$$
\left|\int_{s_1}^{T}e^{-i\phi}\dd t\right|=\left|\int_{s_1}^{T}e^{2i(x_{2}-x_{1})(x_{3}-x_{1})(w_{t}-w_{s_1})}\dd t\right|=\left|\Phi^{w}_{s_1T}(2(x_2-x_1)(x_3-x_1))\right|.
$$
Then we have to bound the following integral 
\begin{equation*}
\begin{split}
I(\alpha)&=\int_{\mathbb R^3}\dd x_1\dd x_2\dd x_3|x_2-x_1|^{\alpha}|\Phi^{w}_{s_1T}(2(x_2-x_1)(x_3-x_1))||\hat\psi_{s_1}(x_1)||\hat\psi^*_{s_2}(x_2)||(\hat\psi_{s_3})^*(x_3)||(\hat\psi^*_{s_4})^*(x_4)|.
\end{split}
\end{equation*}
An application of Lemma~\ref{lemma=Sch-estimates} shows that
$
\mathcal I\lesssim T^{\gamma}(\int_{0}^{T}|\psi_{s}|_{L^{2}(\mathbb R)}\dd s)^4
$
what concludes the proof.
\end{proof}
The following  Gagliardo-Nirenberg-type inequality allows to transform the regularity gain of the previous proposition into an integrability result of Strichartz's type.
\begin{lemma}\label{lemma:Gagliardo-Nirenberg}
Let $p>2$ and $\eps>0$ then there exist $C=C(\eps,p)$ such that for all $f\in L^1(\mathbb R)\cap\mathcal H_{s}$ the following inequality holds :
$$
||f||_{L^p(\mathbb R)}\leq C||f||_{L^1(\mathbb R)}^{1-\theta}||f||_{\mathcal H_s}^{\theta}
$$
where $\mathcal H_s$ is the homogenous Sobolev space on $\RR$, $s=\frac{1}{2}-\frac{1}{p}+\frac{\eps}{2}$ and $\theta=\frac{2p-2}{(2+\eps)p-2}\in(0,1)$ 
\end{lemma}
\begin{proof}
We begin by decomposing $f$ in standard Littlewood-Paley blocks $f=\sum_{i>-1}\Delta_if$ and then
\begin{equation}\label{eq:L-P-decomp}
||f||_{L^p(\mathbb R)}\leq||\Delta_{-1}f||_{L^{p}(\mathbb R)}+\sum_{i\geq0}||\Delta_if||_{L^p(\mathbb R)}.
\end{equation}
Bernstein's inequality then gives
$
||\Delta_if||_{L^p(\mathbb R)}\lesssim2^{i(\frac{1}{2}-\frac{1}{p})}||\Delta_if||_{L^2(\mathbb R)}
$ and
then summing this last equation over $i\geq0$ and using Jensen inequality we can see that 
\begin{equation*}
\begin{split}
\sum_{i\geq0}||\Delta_if||_{L^p(\mathbb R)}&\lesssim \sum_{i\geq0}2^{i(\frac{1}{2}-\frac{1}{p})}||\Delta_if||_{L^2(\mathbb R)}
\lesssim \sum_{i\geq0}(2^{i(\frac{1}{2}-\frac{1}{p}+\frac{\eps}{2})}||\Delta_if||_{L^2(\mathbb R)})2^{-i\frac{\eps}{2}}
\\&\lesssim_{\eps} \left(\sum_{i\geq0}2^{2i(\frac{1}{2}-\frac{1}{p}+\frac{\eps}{2})}||\Delta_if||^2_{L^2(\mathbb R)}\right)^{1/2}
\lesssim_{\eps}||f||_{\mathcal H_s}
\end{split}
\end{equation*}
and then we have 
$
||f||_{L^p(\mathbb R)}\lesssim ||f||_{L^1(\mathbb R)}+||f||_{\mathcal H_s}
$.
Now if we put $f_{\lambda}(x)=f(\lambda x)$ in this last inequality we can see that 
$$
||f||_{L^p(\mathbb R)}\lesssim \lambda^{-1+1/p}||f||_{L^1(\mathbb R)}+\lambda^{\eps/2}||f||_{\mathcal H_s}
$$
then to have our result is suffices to take $\lambda=(||f||^{-1}_{L^1(\mathbb R)}||f||_{\mathcal H_s})^{\frac{2p}{2-(2+\eps)p}}$.
\end{proof}
\begin{proof}[Proof of Theorem~\ref{theorem:strichartz-w}]
Starting with Lemma~\ref{lemma:Gagliardo-Nirenberg} and taking $\alpha=1$ and $\rho>1$ in  Th.~\ref{proposition: Strich} we obtain :
\begin{equation*}
\begin{split}
& \left|\left|\int_{0}^{.}U^{w}_{.}(U_{s}^{w})^{-1}\psi_s\dd s\right|\right|^{p}_{L^{p}([0,T],L^{2p}(\mathbb R))}=\int_{0}^{T}\dd t\left|\left|\left|\int_{0}^{t}U^{w}_{t}(U_{s}^{w})^{-1}\psi_s\dd s\right|^2\right|\right|^{p/2}_{L^p(\mathbb R)}
\\&\qquad\lesssim \int_{0}^{T}\dd t\left|\left|\left|\int_{0}^{t}U^{w}_{t}(U_{s}^{w})^{-1}\psi_s\dd s\right|^2\right|\right|^{1/2}_{L^1(\mathbb R)}\left|\left|D^{\frac{1}{2}}\left|\int_{0}^{t}U^{w}_{t}(U_{s}^{w})^{-1}\psi_s\dd s\right|^2\right|\right|^{\frac{p-1}{2}}_{L^2(\mathbb R)}
\\&\qquad\lesssim C_{w}T^{\frac{\gamma(p-1)}{4}+\frac{5-p}{4}}\left|\left|\int_{0}^{.}U_{.}^{w}(U^{w}_{s})^{-1}\psi_s\dd s\right|\right|_{L^{\infty}([0,T],L^2(\mathbb R))}\left(\int_{0}^{T}||\psi_s||_{L^2(\mathbb R)}\dd s\right)^{p-1}.
\end{split}
\end{equation*}
Now is suffice to remark that 
\begin{equation*}
\begin{split}
\left|\left|\int_{0}^{.}U_{.}^{w}(U^{w}_{s})^{-1}\psi_s\dd s\right|\right|_{L^{\infty}([0,T],L^2(\mathbb R))}&\leq\sup_{0\leq t\leq T}\int_{0}^{T}\left|\left|U^{w}_{t}(U^{w}_{s})^{-1}\psi_s\right|\right|_{L^{2}(\mathbb R)}\dd s
\leq\int_{0}^{T}||\psi_s||_{L^2(\mathbb R)}\dd s
\end{split}
\end{equation*}
and then  
$$
\left|\left|\int_{0}^{.}U^{w}_{.}(U_{s}^{w})^{-1}\psi_s\dd s\right|\right|^p_{L^{p}([0,T],L^{2p}(\mathbb R))}\leq C_{w}T^{\frac{\gamma(p-1)}{4}+\frac{5-p}{4}}\left(\int_{0}^{T}||\psi_s||_{L^2(\mathbb R)}\dd s\right)^{p}
$$
when $p\in(2,5]$. Now in the case $p\in[2,4)$ we obtain the same result if we use the Lemma~\ref{lemma:Gagliardo-Nirenberg} and take $\alpha=1-\frac{2}{p}+\eps\in(0,1/2]$, $\rho>\alpha+\frac{1}{2}$ in Th.~\ref{proposition: Strich}. 
\end{proof}
To have all the ingredients needed for the fixed point argument we have to estimates the action of the operator
$U^{w}$ on the initial condition.
\begin{proposition}\label{proposition:strich-estim-3}
Let $T>0$, $p=\mu+1\in(4,5]$, $\rho>\min(\frac{3}{2}-\frac{2}{p})$ then then there exist constant $C_{p}$  and $\gamma^{\star}(p)>0$ such that the following inequality holds :
$$
\left|\left|U^{w}_{t}\psi\right|\right|_{L^{p}([0,T],L^{2p}(\mathbb R))}\leq C_{p}||\Phi||_{\mathcal W^{\rho,\gamma}_T}T^{\gamma^{\star}(p)}||\psi||_{L^2(\mathbb R)}
$$
for all $\psi\in L^{2}(\mathbb R)$.
\end{proposition}
\begin{proof}
Let us begin by using the Lemma~\ref{lemma:Gagliardo-Nirenberg} and then 
\begin{equation*}
\begin{split}
\left|\left|U^{w}_{t}\psi\right|\right|^p_{L^{p}([0,T],L^{2p}(\mathbb R))}&=\int_{0}^{T}\dd t\left|\left|\left|U^{w}_t\psi\right|^2\right|\right|_{L^p(\mathbb R)}^{p/2}
\lesssim||U^{w}\psi||^{(1-\theta)p}_{L^{\infty}([0,T],L^{2}(\mathbb R))}\int_{0}^{T}\dd t\left|\left|D^{\alpha/2}\left|U^{w}_t\psi\right|^2\right|\right|_{L^{2}(\mathbb R)}^{\theta \frac{p}{2}}
\end{split}
\end{equation*}
where $\theta=\frac{2p-2}{(2+\eps)p-2}$ then it suffices to bound the quantity 
$\|D^{\alpha/2}\left|U^{w}_t\psi\right|^2\|^2_{L^2}$ and to proceed as the in the Proposition~\ref{theorem:strichartz-w}. By a simple computation we have 
\begin{equation*}
\begin{split}
\left|\left|D^{\alpha/2}\left|U^{w}_t\psi\right|^2\right|\right|^2=\int_{\mathbb R^3}\dd x_1\dd x_2\dd x_3|x_2-x_1|^{\alpha}|\Phi^{w}_{0T}(\eta)||\hat\psi(x_1)||\hat\psi(x_2)||\hat\psi(x_3)||\hat\psi(-x_1+x_2+x_3)|
\end{split}
\end{equation*}
where $\eta=2(x_2-x_1)(x_3-x_2)$.
Now applying again Proposition~\ref{proposition: Strich} we concludes the proof.
\end{proof}
We are now ready to prove Theorem~\ref{th:Strich-existence} about existence of local solution to the modulated NLS with general non-linearity.
\begin{proof}[Proof of Theorem~\ref{th:Strich-existence}]
In this proof we use essentially the the same argument of~\cite{Debussche2011363} to obtain the local and global existence. Let us define for $\psi\in L^{p}([0,T],L^{2p})$ the following map :
$$
\Gamma(\psi)_t=U^{w}_tu^{0}+i\int_{0}^{t}U^{w}_t(U^{w}_{s})^{-1}(|\psi_s|^{\mu} \psi_s)\dd s.
$$
We will prove that $\Gamma $ is a strict contraction in a suitable ball of $L^{p}([0,T],L^{2p}(\mathbb R))$. Let
$$
B_{r}=\left\{\psi\in L^{p}([0,T],L^{2p}(\mathbb R)), ||\psi||_{L^{p}([0,T],L^{2p}(\mathbb R))}\leq r\right\},
$$
using Proposition~\ref{proposition:strich-estim-3} and Theorem~\ref{theorem:strichartz-w} we have 
\begin{equation*}
\begin{split}
\left|\left|\Gamma(\psi)\right|\right|_{L^{p}([0,T],L^{2p}(\mathbb R))}&\leq||U_{.}^{w}u^{0}||_{L^{p}([0,T],L^{2p}(\mathbb R))}+\left|\left|\int_{0}^{.}U^{w}_{.}(U^{w}_{s})^{-1}(|\psi_s|^{\mu}\psi_s)\dd s\right|\right|_{L^{p}([0,T],L^{2p}(\mathbb R))} 
\\&\leq C_{w,T}T^{\gamma^{\star}(p)}\left(||u^0||_{L^2(\mathbb R)}+\int_{0}^{T}\left|\left|\left|\psi_s\right|^{\mu}\psi_s\right|\right|_{L^{2}(\mathbb R)}\dd s\right)
\\&\leq C_{w,T}T^{\gamma^{\star}(p)} (||u^{0}||_{L^2(\mathbb R)}+||\psi||_{L^{p}([0,T],L^{2p}(\mathbb R))}^p)
\\&\leq C_{w,T}T^{\gamma^{\star}(p)}(||u^0||_{L^{2}(\mathbb R)}+r^{p})
\end{split}
\end{equation*}
then we can choose $T_1$ small enough such for all $T\leq T_1$   the equation $r_T=C_{w}T^{\gamma^{\star}(p)}(||u^0||_{L^{2}(\mathbb R)}+r_T^{p})
$ admits a positive solution $r=r_T$. Now for $T<T_1$ and  $\psi_1,\psi_2\in L^{p}([0,T],L^{2p}(\mathbb R))\cap B_r$ we see by the same argument used previously that we have 
\begin{equation*}
\begin{split}
\left|\left|\Gamma(\psi_1)-\Gamma(\psi_2)\right|\right|_{L^{p}([0,T],L^{2p}(\mathbb R))}&=\left|\left|\int_{0}^{.}U^{w}_{.}(U^{w}_s)^{-1}(\left|\psi_1\right|^{\mu}\psi_1-\left|\psi_2\right|^{\mu}\psi_2)\dd s\right|\right|_{L^{p}([0,T],L^{2p}(\mathbb R))}
\\&\leq C_{w,T}T^{\gamma^{\star}(p)}r^{p-1}\left|\left|\psi_1-\psi_2\right|\right|_{L^{p}([0,T],L^{2p}(\mathbb R))}.
\end{split}
\end{equation*}
Choosing $T_2<T_1$ such that for all $T<T_2$ we have $||\Phi||_{\mathcal W^{\gamma,\rho}_{T}}T^{\gamma^{\star}(p)}r^{p-1}<1$ which implies that $\Gamma$ is a strict contraction in the ball $L^{p}([0,T_2],L^{2p}(\mathbb R))\cap B_r$ and therefore has a unique fixed point in this ball. The proof of uniqueness is standard. Now the fact that $u\in C([0,T_2],L^2(\mathbb R))$ is simply given by the inequality
$$
||u_t-u_s||_{L^2(\mathbb R)}\leq||(U^w_t-U_s^w)u^0||_{L^2(\mathbb R)}+\int_{s}^{t}||u_\sigma||_{L^{2p}(\mathbb R)}^p\dd\sigma .
$$
Now we will focus on the proof of the conservation law in the quintic case (i.e.: $p=5$), for simplicity. For values of $1\le p<5$ the argument is similar. Let now $M\in\mathbb N$. By the same argument used in the beginning  of the proof  we can construct a local solution $u^M\in L^5([0,T],L^{10}(\mathbb R))\cap C([0,T],L^2(\mathbb R))$ of the regularized equation. More precisely we have that 
$$
u^M_t=U_t^w\Pi_Mu^0+i\int_{0}^{t}U_t^w(U_s^w)^{-1}\Pi_M\left(|\Pi_Mu^M|^4\Pi_Mu^M\right)\dd s
$$
for all $t\in[0,T]$ and some $T=T(||u^0||_{L^2(\mathbb R)})$. Let $v^M=(U^w_t)^{-1}u_t^M$. A simple computation shows that
$$
||v_t^M-v_s^M||_{L^2(\mathbb R)}\leq\int_{s}^{t}||\Pi_Mu_{\sigma}^{M}||^5_{L^{10}}\dd\sigma\lesssim_M(t-s)||u^M||_{L^{\infty}([0,T],L^{2}(\mathbb R))}
$$
from which  we obtain that 
$$
||v_t^{M}||_{L^2(\RR)}^2=||v^{M}_s||_{L^2(\RR)}^2+2\int_{s}^{t}\mathcal Im\langle v_s, (U^w_\sigma)^{-1}(|U^w_\sigma v_s^M|^4U_\sigma v_s^M)\rangle\dd \sigma+O(|t-s|^2).
$$
It is not difficult to see that $\langle v_s, (U^w_\sigma)^{-1}(|U^w_\sigma v_s^M|^5U_\sigma v_s^M)\rangle\in\mathbb R$ and then $||v^M_t||_{L^2(\RR)}^2=||v^M_s||_{L^2(\RR)}^2+O(|t-s|^2)$ which implies that $||u^M_t||_{L^2(\RR)}=||u^0||_{L^2(\RR)}$. Moreover we have  
\begin{itemize}
\item $\Pi_Mu^M=u^M$ ;
\item for every $T>0$, $\sup_M||u^M||_{L^5([0,T],L^{10}(\mathbb R))}<+\infty$ .
\end{itemize}
Using that we have easily 
$$
||u^M-u||_{L^5([0,T],L^{10}(\mathbb R))}\lesssim T^{\gamma}(||u^0-\Pi_Mu^0||_{L^2}+||u^M-u||_{L^5([0,T],L^{10}(\mathbb R))})
$$
and for $T<\min(T_2,1/2)$ we have that $||u^M-u||_{L^5([0,T],L^{10}(\mathbb R))}\to^{M\to+\infty}0$. It is then  sufficient  to iterate this procedure to extend it to the interval $[0,T_2]$. Now by a simple computation we can see that 
$$
||u^M-u||_{L^\infty([0,T_2],L^2(\mathbb R))}\lesssim ||\Pi_Mu^0-u^0||_{L^2(\RR)}+ ||u^M-u||_{L^5([0,T],L^{10}(\mathbb R))}
$$
and then $||u_t||_{L^2(\mathbb R)}=||u^0||_{L^2(\mathbb R)}$ which gives the conservation law and allow us to extend our local solution in a global solution. Now let $u^0\in H^1$ and using the Strichartz estimates after taking the first derivative of the function $\Gamma(\psi)$ we obtain that 
$$
||\Gamma(\psi)||_{L^5([0,T],W^{1,10}(\mathbb R))}\lesssim_{w}T^{\gamma}(||u^0||_{H^1(\RR)}+r^4||\psi||_{L^5([0,T],W^{1,10}(\mathbb R))})
$$  
with $\psi\in B_r$ where $B_r$ is the ball in which we have setup our point fix argument at the beginning of the proof. Then $B(0,R)$, the ball of radius $R$ in $L^5([0,T],W^{1,10}(\mathbb R))$ is invariant  by $\Gamma_{B_r}$ the restriction of $\Gamma$ on $B_r$  for $T_3=T$ depending only on $r$ and not $R$. Since closed balls of $L^5([0,T],W^{1,10}(\mathbb R))$ are closed also in $L^5([0,T],L^{1,10}(\mathbb R))$  the fixed point of $\Gamma_{B_r}$ is in $L^5([0,T],W^{1,10}(\mathbb R))$  and  we obtain  that $u\in L^5([0,T_3],W^{1,10}(\mathbb R))$. Now an argument similar to that used above in the $L^2$ case gives that  $u\in C([0,T_3],H^{1}(\mathbb R))$ since $T_3$ depends only on $r$, by using the conservation of the $L^2$ norm, we can repeat the argument in the interval $[T_3, 2T_3]$ and so on to conclude that $u\in C(\RR^+,H^{1}(\mathbb R))$.
\end{proof}
The following lemma show that our solution correspond to that obtain by  de~Bouard and Debussche in~\cite{deBouard20101300}.
\begin{lemma}
\label{lemma:reg}
Let $u$ the solution constructed in the Theorem~\ref{th:Strich-existence} then $u\in L^{p+1}_{loc}([0,+\infty),L^{p+1}(\mathbb R))$. In particular if $w$ is a Brownian motion, the solution $u$ coincides with that studied by De~Bouard and Debussche in~\cite{deBouard20101300}. 
\end{lemma}
\begin{proof}
Simply use the Cauchy-Schawrtz inequality :
$$
||u_t||^{p+1}_{L^p(\mathbb R)}\leq||u_t||_{L^2(\mathbb R)}||u_t||^p_{L^{2p}(\mathbb R)}
$$
and then integrating this inequality for $t\in[0,T]$ and using the $L^2$ conservation law get 
$$
||u||^{p+1}_{L^{p+1}([0,T],L^{p+1}(\mathbb R))}\leq||u^0||_{L^2(\mathbb R)}||u||^p_{L^{p}([0,T],L^{2p}(\mathbb R))}
$$
and thus $\in L^{p+1}_{loc}([0,+\infty),L^{p+1}(\mathbb R))$.
Now when $w$ is a Bronwian motion, $u^0\in L^2(\mathbb R)$ and $\sigma<2$ let us recall that in~\cite{deBouard20101300} the authors show the existence of a unique solution $v$ of the equation 
\begin{equation}
\label{eq:dd}
v(t)=U_t^wu^0+\int_0^tU^w_t(U^w_s)^{-1}(|v_s|^{2\sigma}v_s)\dd s
\end{equation}
which satisfy $v\in L^{2\sigma+2}_{loc}([0,+\infty),L^{2\sigma+2}(\mathbb R)$ almost surely. Taking $p=2\sigma+1$ we can see with this notation that the solution $u$ constructed in Theorem~\ref{th:Strich-existence} satisfies the equation~\eqref{eq:dd} almost surely and according to the Lemma~\ref{lemma:reg} it satisfies also  $u\in L^{2\sigma+2}_{loc}([0,+\infty),L^{2\sigma+2}(\mathbb R)$ which  by uniqueness implies that the two solution coincide (i.e.: $u=v$).
\end{proof}

\end{document}